\documentclass[10pt]{article}
\usepackage[utf8]{inputenc}
\usepackage{graphicx}
\usepackage{amsmath}
\allowdisplaybreaks[4]
\usepackage{amssymb}
\numberwithin{equation}{section}

\rmfamily

\usepackage[top = 2.5cm, bottom = 2.5cm, left = 2.5cm, right = 2.5cm]{geometry}
\usepackage[T1]{fontenc}
\usepackage[utf8]{inputenc}
\usepackage{multirow}
\usepackage{booktabs}
\usepackage{graphicx}
\graphicspath{ {images/} }
\usepackage{setspace}
\usepackage{float}
\usepackage{mathtools}
\mathtoolsset{showonlyrefs}
\usepackage{fancyhdr}

\usepackage{amsmath}
\newcommand{\T}{\mathbb{T}}
\newcommand{\R}{\mathbb{R}}
\newcommand{\Z}{\mathbb{Z}}
\newcommand{\C}{\mathbb{C}}
\newcommand{\Sph}{\mathbb{S}}
\newcommand{\Pcal}{\mathcal{P}}
\newcommand{\Scal}{\mathcal{S}}
\newcommand{\K}{\mathbb{K}}

\newcommand{\paren}[1]{\left( #1 \right)}

\newcommand{\abs}[1]{\left| #1 \right|}
\newcommand{\norm}[1]{\left\| #1 \right \|}
\usepackage[spanish,english]{babel}
\usepackage[colorlinks=true,linkcolor=blue,citecolor=red,urlcolor=blue]{hyperref}

\usepackage{amsthm}
\renewenvironment{proof}[1][\proofname]{\par
  \pushQED{\qed}%
  \normalfont\textit{#1.}~\ignorespaces
}{%
  \popQED\endtrivlist
}
\renewcommand{\qedsymbol}{$\blacksquare$}

\newtheorem{teorema}{Theorem}[section]

\newtheorem{corolario}[teorema]{Corollary}

\theoremstyle{definition}

\theoremstyle{remark}
\newtheorem{remark}[teorema]{Remark}

\addto\captionsenglish{}
\addto\captionsenglish{}
\usepackage{hyperref}
\usepackage{cleveref}

\begin{document}

\thispagestyle{empty}
\begin{titlepage}
\vspace{1cm}

\begin{center}
     {\LARGE \textbf{On Periodic Solutions of the
Zakharov–Rubenchik System}}
    
    \vspace{2cm}
    
    {\large \textbf{Yeison Alejandro Gómez Hernández}} \\
    Departamento de Matemáticas \\
    Universidad Nacional de Colombia \\
    Manizales, Colombia
    
    \vspace{1cm}
    
    \textbf{and}
    
    \vspace{1cm}
    
    {\large \textbf{Juan Carlos Cordero Ceballos}} \\
    Departamento de Matemáticas \\
    Universidad Nacional de Colombia \\
    Manizales, Colombia
\end{center}

    \vspace{2cm}

   \begin{center}
       {\Large \bf Abstract}
   \end{center}
 \vspace{1cm}
    We study the one-dimensional Zakharov-Rubenchik system
\begin{equation*}\label{ZR abstract}
    \left\{\begin{array}{l}
\partial_t\psi-\sigma_3 \partial_x\psi-i \delta \partial_x^2\psi+i\left\{\sigma_2|\psi|^2+W\left(\rho+D \partial_x\phi\right)\right\} \psi=0, \\
\partial_t\rho+\partial_x^2\phi+D\partial_x\left(|\psi|^2\right)=0, \\
\partial_t\phi+\frac{1}{M^2} \rho+|\psi|^2=0,
\end{array}\right.
\end{equation*}
with periodic initial conditions $(\psi_0, \rho_0, \phi_0)$. This system models the interaction of low-frequency waves with high-frequency waves in different physical phenomena. Using semigroup theory, we prove the global well-posedness of the associated linear equation in $H^q(\mathbb{T})\times H^s(\mathbb{T})\times H^{s+1}(\mathbb{T})$, $q,s \in \mathbb{R}$. Subsequently, we consider the modified Zakharov-Rubenchik system \eqref{ZR modificado} and, by means of a change of variables and the Banach Fixed Point Theorem, we obtain local well-posedness in the space $H^s(\mathbb{T})\times H^s(\mathbb{T})\times H^{s+1}(\mathbb{T})$, $s>3/2$.
\vspace{1cm}

\end{titlepage}

\thispagestyle{empty}
\begin{titlepage}
   \begin{center}
       {\Large \bf Introduction}
   \end{center}
    \vspace{1cm}
    This paper deals with the well-posedness of the Cauchy problem associated with the Zakharov-Rubenchik system (known as the Benney-Roskes system in the context of water waves), with periodic initial conditions. In \cite{zakharov1972nonlinear} this system was derived by V. E. Zakharov and A. M. Rubenchick to describe the interaction of any type of high-frequency waves with acoustic-type waves.
    \\

    In the notation of \cite{saut2005well} and \cite{luong2018cauchy}, the Zakharov-Rubenchik/Benney-Roskes (ZR/BR) system has the following form
    \begin{equation}\label{ZR introducción 2 o 3 dimensiones}
    \left\{\begin{array}{l}
\partial_t\psi-\sigma_3 \partial_x\psi-i \delta \partial_x^2\psi-i\sigma_1\Delta_{\perp}\psi+i\left\{\sigma_2|\psi|^2+W\left(\rho+D \partial_x\phi\right)\right\} \psi=0, \\
\partial_t\rho+\Delta\phi+D\partial_x\left(|\psi|^2\right)=0, \\
\partial_t\phi+\frac{1}{M^2} \rho+|\psi|^2=0,
\end{array}\right.
\end{equation}

    where \(\psi : \mathbb{R} \times \mathbb{R}^d \to \mathbb{C}\) is the wave function, with fast oscillation, and \( \rho, \phi : \mathbb{R} \times \mathbb{R}^d \to \mathbb{R}\) describe acoustic-type waves ($d=2,3$). The functions $\psi, \rho$ and $\phi$ depend on the time variable $t\in \mathbb{R}$ and $\vec{x}\in \mathbb{R}^d$ ($\vec{x}=(x,z)$ and $\Delta_{\perp}=\partial_x^2$ if $d=2$; $\vec{x}=(x,y,z)$ and $\Delta_{\perp}=\partial_x^2+\partial_y^2$ if $d=3$. Clearly $\Delta=\Delta_{\perp}+\partial_z^2$). 
\\

Here \(\sigma_1, \sigma_2, \sigma_3 = \pm 1\), $W>0$ measures the intensity of the coupling with the acoustic-type waves, \(M=|v_g|/c_s \) is known as the Mach number and , ($v_g$ is the group velocity of the carrier wave and $c_s$ is the speed of sound), \(D \in \mathbb{R}\) is associated with the Doppler shift due to the velocity of the medium, and \(\delta \in \mathbb{R}\) is a dimensionless dispersion coefficient.
\\

The local well-posedness of \eqref{ZR introducción 2 o 3 dimensiones} in \( H^s(\mathbb{R}^d) \times H^{s-1/2}(\mathbb{R}^d) \times H^{s+1/2}(\mathbb{R}^d) \) with \( s > \frac{d}{2}, d = 2, 3 \), was obtained in \cite{saut2005well}, using the local smoothing property of the free Schrödinger operator after reducing the system to a (nonlocal) quasilinear Schrödinger equation. When dealing with the linear Schrödinger equation, with local operator $\Delta$, it is known that the solution satisfies the estimate $t^{d/2}\Vert \psi(\cdot,t) \Vert_{L^\infty_x}\lesssim \Vert \psi(\cdot,0)\Vert_{L^1_x} $, that is, for a well-localized (nonzero) initial data the solution disperses as $t\to\infty$ at rate $t^{-d/2}$. Dispersion occurs because plane waves (fronts sharing the same phase are parallel planes) with different wave numbers have different phase velocities. In \cite{LinaresPonce2015} one can find more precise mathematical details about the dispersive effects associated with the solution of the linear Schrödinger equation. Since the work in \cite{saut2005well} uses dispersive properties of the free Schrödinger group that are valid only on the whole space, the proof does not extend to the Cauchy problem posed on \(\mathbb{T}^d\) or \(\mathbb{R}^{d-1} \times \mathbb{T}\). 
\\

On the other hand, when this method is applied to the Benney-Roskes system, it provides an existence time of order \( O(1) \); however, an existence time of order \( O(\epsilon^{-1}) \) is needed to fully justify the Benney-Roskes system as a water wave model on the correct time scales (see \cite{Lannes2006} for further details).
\\

The local well-posedness of the Zakharov-Rubenchik/Benney-Roskes system was also obtained in Obrecht's thesis \cite{obrecht2015approximation}, in the space $H^{s+1}(\mathbb{R}^d) \times H^{s}(\mathbb{R}^d) \times H^{s+1}(\mathbb{R}^d)$ \( s > d/2+1 \), with the additional condition \( \delta\sigma_1 > 0 \) (that is, when the second-order operator in the first equation of \eqref{ZR introducción 2 o 3 dimensiones} is elliptic). This was done using an energy method, following the ideas of Schochet-Weinstein in \cite{SchochetWeinstein1986} for the nonlinear Schrödinger limit of the Zakharov system. The method used in \cite{SchochetWeinstein1986} and \cite{obrecht2015approximation} consists of rewriting the Zakharov system (or the Zakharov-Rubenchik system) as a dispersive (anti-self-adjoint) perturbation of a symmetric nonlinear hyperbolic system, and then exploiting the algebraic structure of the system. Later, Luong and Saut \cite{luong2018cauchy} use the same method to deduce existence and uniqueness in two dimensions, including the aforementioned small parameter $\epsilon$ and a divergence condition that prepares the data. This result can be read as follows:\\

   The system
\begin{equation}\label{ZR introducción 2 dimensiones con epsilon}
    \left\{\begin{array}{l}
\partial_t\psi-\sigma_3 \partial_x\psi-i \epsilon\delta \partial_x^2\psi-i\epsilon\sigma_1\partial_y^2\psi+i\epsilon\left\{\sigma_2|\psi|^2+W\left(\rho+D \partial_x\phi\right)\right\} \psi=0, \\
\partial_t\rho+\Delta\phi+D\partial_x\left(|\psi|^2\right)=0, \\
\partial_t\phi+\frac{1}{M^2} \rho+|\psi|^2=0,
\end{array}\right.
\end{equation} 
with $\delta\sigma_1>0$ and initial data
\[
(\psi_0,\rho_0,\phi_0)\in H^{s+1}(\mathbb{R}^2)\times H^{s}(\mathbb{R}^2)\times H^{s+1}(\mathbb{R}^2)
\]
satisfying
\begin{equation}
WM\!\left(-\Delta\phi_0-\frac{D}{M^2}\partial_x\rho_0\right)
= \nabla\cdot V_0,
\label{eq:compatibilidad}
\end{equation}
for some $V_0\in H^{s}(\mathbb{R}^2)^2$ and $s>2$, has a unique solution 
\[
(\psi,\rho,\phi)\in
L^\infty\!\bigl(0,T;H^{s+1}(\mathbb{R}^2)\bigr)
\times
L^\infty\!\bigl(0,T;H^{s}(\mathbb{R}^2)\bigr)
\times
L^\infty\!\bigl(0,T;H^{s+1}(\mathbb{R}^2)\bigr)
\]
for some $T>0$ independent of $\epsilon$.\\

Their method leads them to state that, with minor changes, the same result holds in the three-dimensional case, that is, by replacing $\partial_y^2\psi$ with $\nabla_\perp \psi$, and also in the periodic case $\mathbb{T}^d$ ($d=2,3$) or in the semi-periodic case $\mathbb{R}^{d-1}\times\mathbb{T}$. However, this method also does not provide the expected time scale $O(\epsilon^{-1})$ when applied to the Benney-Roskes system (see \cite{obrecht2015approximation}, chapter 3). 
\\

The ``dispersive method'' used in \cite{saut2005well} works in \(\mathbb{R}^d\) but does not require the Schrödinger part of the system to be ``elliptic'' (that is, it does not require the condition \( \delta\sigma_1 > 0 \)). Moreover, it reduces the regularity of the initial data and could also be applied to (possibly unphysical) nonlinear perturbations of the system. On the other hand, Schochet-Weinstein-type methods allow one to treat the periodic or semiperiodic cases, but are relatively rigid (they depend on the algebraic structure of the system) and require initial data in a space of higher regularity  $H^{s+1}(\mathbb{R}^d)\times H^{s}(\mathbb{R}^d)\times H^{s+1}(\mathbb{R}^d), s > \frac{d}{2} + 1$ .
\\

It is worth noting that the system \eqref{ZR introducción 2 o 3 dimensiones} possesses two conserved quantities, one is the $L^2$ norm of $\psi$, usually called the mass
$$
\mathcal{M}(t)=\int |\psi(x,y_\perp,t)|^2 ,
$$

where $y_\perp = y$ or $(y,z)$, and the other is the energy of the system

$$\mathcal{E}(t) = \int_{\mathbb{R}^d}
\left(
\frac{\delta}{2}|\psi_x|^2
+ \frac{\sigma_1}{2}|\nabla_\perp \psi|^2
+ \frac{\sigma_2}{4}|\psi|^4
+ \frac{W}{4M^2}\rho^2
+ \frac{W}{4}|\nabla \phi|^2
+ \frac{W}{2}(\rho + D\phi_x)|\psi|^2
\right),
$$
once the change of variable $(x,t) \to (x+\sigma_3 t,t)$ is applied. These quantities can reveal and confirm a Hamiltonian structure in the previous version of the model.\\ 

Regarding the one-dimensional case, Oliveira \cite{oliveira2003stability} proved local well-posedness (and later global well-posedness using the aforementioned conserved quantities) in \( H^2(\mathbb{R}) \times H^1(\mathbb{R}) \times H^1(\mathbb{R}) \). This result was improved in \cite{linares2009well} where global well-posedness was established in the energy space \( H^1(\mathbb{R}) \times L^2(\mathbb{R}) \times L^2(\mathbb{R}) \).
\\

Obrecht \cite{obrecht2015approximation} also considers the full-dispersion Benney-Roskes system (a version of the Benney-Roskes system that takes into account the dispersion relation of the surface waves, which was not previously accounted for). She modifies this system by adding terms of order $O(\epsilon)$ in the second and third equations of the system. According to the rigorous derivation of the equations satisfied by the functions (see \cite{Lannes2013}, \cite{obrecht2015approximation}), this modification can be made in the one-dimensional case without changing the order of consistency of the system. Finally, this modification allows one to deduce well-posedness on the correct time scale $O(\epsilon^{-1})$.
\\

Regarding the dynamics in certain parameters and time of solutions of the ZR/BR system, some results can be found in \cite{CorderoCeballos2016}, \cite{Luong2025} and \cite{MartinezPalacios2021}. 
\\

In the present paper we use semigroup theory to prove the global well-posedness of the linear equation associated with the ZR/BR system in the space $H^q(\mathbb{T})\times H^s(\mathbb{T})\times H^{s+1}(\mathbb{T})$, $q,s\in \mathbb{R}$. Since we cannot use the smoothing property of the Schrödinger group (we are working in the periodic case), the proof does not extend to the full system. Inspired by \cite{obrecht2015approximation} and \cite{luong2018cauchy}, we consider the system \eqref{ZR modificado} with initial conditions $(\psi_0, \rho_0,\phi_0)$ in $H^s(\mathbb{T})\times H^s(\mathbb{T})\times H^{s+1}(\mathbb{T})$, $s>3/2$, and prove local well-posedness using the Banach Fixed Point Theorem. The argument achieves an existence time of order $O(\epsilon^{-1})$. 
\\

On the following page you will find the notation that we will use throughout these notes, all of which is commonly used in the field of analysis.  

\end{titlepage}

\thispagestyle{empty}
\begin{titlepage}
   \begin{center}
       {\Large \bf Basic Notation}
   \end{center}
    \vspace{1cm}
    \begin{itemize}
    \item $\T = \R/2\pi\Z \cong \Sph^{1} = \{z \in \C : |z| = 1\}$.
    
    \item $\partial_xf$ denotes the partial derivative of the function $f$ with respect to the variable $x$.
    \item $\Pcal = C^{\infty}(\T)$ is the space of infinitely differentiable functions on $\T$ with complex values. 
    
    \item $\Pcal^{\prime}$ denotes the space of periodic distributions.
    \item $\Scal(\R)$ is the Schwartz space of smooth rapidly decaying functions and $\Scal^{\prime}(\R)$ is the space of tempered distributions.
    \item $\widehat{\cdot}$ or $\mathcal{F}$  denote the Fourier transform.
    \item $(\cdot)^{\vee}$ or $\mathcal{F}^{-1}$  denote the inverse Fourier transform.
    \item $\langle \cdot, \cdot \rangle_{X}$ denotes the inner product in the Hilbert space $X$.
    \item $\langle \cdot, \cdot \rangle$ denotes the scalar product for the duality of $X'$, $X$, that is, $\langle f, x \rangle = f(x)$, for each $f \in X'$ and $x \in X$.
    \item For $X$ and $Y$ Banach spaces, $X \hookrightarrow Y$ means that the space $X$ is continuously and densely embedded in $Y$.
    \item $l^{p}(\Z) = \{\alpha = (\alpha(k))_{k \in \Z} \subset \C \mid \sum_{k \in \Z} |\alpha(k)|^{p} < \infty\}$. The norm of this space is denoted by $\|\alpha\|_{l^{p}(\Z)} = \|\alpha(k)\|_{l^{p}(\Z)} = \sum_{k \in \Z} |\alpha(k)|^{p}$.
    \item $l^{\infty}(\Z) = \{\alpha = (\alpha(k))_{k \in \Z} \subset \C \mid \sup_{k \in \Z} \{|\alpha(k)|\} < \infty\}$. The norm of this space is denoted by $\|\alpha\|_{l^{\infty}(\Z)} = \|\alpha(k)\|_{l^{\infty}(\Z)} = \sup_{k \in \Z} \{|\alpha(k)|\}$.
    \item Given the measurable space $(X, \mathbb{X}, \mu)$, we denote by $L^{p}(X)$ the space of all measurable functions $f: X \rightarrow \C$, such that $\int_{X} |f(x)|^{p} d\mu(x) < \infty$ if $1 \leq p < \infty$ or $\text{ess sup}_{X} |f| < \infty$, if $p = \infty$. We will understand $\|\cdot\|_{0} = \|\cdot\|_{L^{2}(X)}$.
    \item $H^{s}(\K)$ is the Sobolev space over $\K = \T$ or $\K = \R$ with index $s \in \R$, of $L^{2}$ type.
    \item $(f, g)_{s}$ denotes the inner product of $f, g \in H^{s}(\K)$, for all $s \in \R$, with $\K = \T$ or $\K = \R$.
   
    \item $\|f\|_{X}$ denotes the usual norm of $f$ in the normed space $X$.
     \item $\|f\|_s$ denotes the norm of $f \in H^s(\K)$, for all $s \in \R$, where $\K = \T$ or $\K = \R$.
    \item $\mathcal{B}(X,Y)$ denotes the space of bounded linear operators from $X$ to $Y$, where $X$ and $Y$ are Banach spaces. If $X = Y$, we denote $\mathcal{B}(X) = \mathcal{B}(X,X)$.
    \item $I_X$ denotes the identity operator on the space $X$.
\end{itemize}
\end{titlepage}

\tableofcontents
\section{The Periodic Zakharov-Rubenchik System}\label{section ZR periodico}

The one-dimensional Zakharov-Rubenchik system has the following form
\begin{equation}\label{ZR}
    \left\{\begin{array}{l}
\partial_t\psi-\sigma_3 \partial_x\psi-i \delta \partial_x^2\psi+i\left\{\sigma_2|\psi|^2+W\left(\rho+D \partial_x\phi\right)\right\} \psi=0, \\
\partial_t\rho+\partial_x^2\phi+D\partial_x\left(|\psi|^2\right)=0, \\
\partial_t\phi+\frac{1}{M^2} \rho+|\psi|^2=0,
\end{array}\right.
\end{equation}
with initial conditions $(\psi_0, \rho_0, \phi_0)$.

\subsection{Global Well-Posedness of the Linear Equation}
The linear system associated with \eqref{ZR} is
\begin{equation}\label{sun-jin}
    \left\{\begin{array}{l}
\partial_t\Psi=A\Psi, \\
\Psi(t=0)=\Psi_0,
\end{array}\right.
\end{equation} 
where
\begin{equation}
    \Psi=\left(\begin{array}{c}
\psi \\
\rho \\
\phi
\end{array}\right),  \quad A=\left(\begin{array}{ccc}
\sigma_3\partial_x+i \delta\partial_x^2 & 0 & 0 \\
0 & 0 & -\partial_x^2 \\
0 & -\frac{1}{M^2}  & 0
\end{array}\right),\quad \text { and }\quad \Psi_0=\left(\begin{array}{c}
\psi_0 \\
\rho_0 \\
\phi_0
\end{array}\right).
\end{equation}
For this system we have the following result, whose proof will be given after some preliminary estimates.

\begin{teorema}\label{Teorema 3.1}
    Let $q,s\in \mathbb{R}$. The system \eqref{sun-jin} is well-posed in $H^q(\mathbb{T})\times H^s(\mathbb{T})\times H^{s+1}(\mathbb{T})$. That is, for every $\Psi_0\in H^q(\mathbb{T})\times H^s(\mathbb{T})\times H^{s+1}(\mathbb{T})$, \eqref{sun-jin} has a unique solution 
    $$\Psi\in C([0,\infty);H^q(\mathbb{T})\times H^s(\mathbb{T})\times H^{s+1}(\mathbb{T}))\cap C^1([0,\infty);H^{q-2}(\mathbb{T})\times H^{s-1}(\mathbb{T})\times H^{s}(\mathbb{T})),$$
    which depends continuously on the initial data.
\end{teorema}

If we apply the Fourier transform to \eqref{sun-jin} we obtain the following Ordinary Differential Equation (ODE) with respect to $t$:
\begin{equation}\label{EDO}
    \left\{\begin{array}{l}
\partial_t\widehat{\Psi(t)}(k)=\mathcal{A}(k)\widehat{\Psi(t)}(k), \quad k\in \mathbb{Z}, t\geq 0 ,\\
\widehat{\Psi(0)}(k)=\widehat{\Psi_0}(k), 
\end{array}\right.
\end{equation} 
where
\begin{equation}
    \mathcal{A}(k)=\left(\begin{array}{ccc}
\sigma_3ik-i \delta k^2 & 0 & 0 \\
0 & 0 & k^2 \\
0 & -\frac{1}{M^2}  & 0
\end{array}\right),
\end{equation}
for every $k \in \mathbb{Z}$. 
\\

Given $q,s,r\in \mathbb{R}$, set $$W:=H^q(\mathbb{T})\times H^s(\mathbb{T})\times H^r(\mathbb{T}), \quad\tilde{W}:=H^{q-2}(\mathbb{T})\times H^{r-2}(\mathbb{T})\times H^s(\mathbb{T}).$$ For $\varphi=(\varphi_1,\varphi_2,\varphi_3) \in \paren{\mathcal{P}'}^3$, we have that 
\begin{equation*}
\begin{split}
     \norm{A\varphi}_{\tilde{W}} & = \left(2\pi\sum_{k=-\infty}^{\infty}(1+k^2)^{q-2} \left| \sigma_3 ik - i \delta k^2 \right|^2 \left| \widehat{\varphi_1}(k) \right|^2 \right)^{1/2} \\
        & \quad + \left(2\pi\sum_{k=-\infty}^{\infty}(1+k^2)^{r-2} \left| k^2 \right|^2 \left| \widehat{\varphi_3}(k) \right|^2 \right)^{1/2} \\
        & \quad + \left(2\pi\sum_{k=-\infty}^{\infty}(1+k^2)^{s} \left| -\frac{1}{M^2} \right|^2 \left| \widehat{\varphi_2}(k) \right|^2 \right)^{1/2}
\end{split}
\end{equation*}

\begin{equation}\label{A}    
\begin{split}
        \quad\quad\quad\quad& \leq \left(2\pi\sum_{k=-\infty}^{\infty}(1+k^2)^{q-2} \left| k \right|^2 \left| \widehat{\varphi_1}(k) \right|^2 \right)^{1/2} \\
        & \quad + \left(2\pi\sum_{k=-\infty}^{\infty}(1+k^2)^{q-2} \left| \delta \right|^2 \left| k \right|^4 \left| \widehat{\varphi_1}(k) \right|^2 \right)^{1/2} \\
        & \quad + \left(2\pi\sum_{k=-\infty}^{\infty}(1+k^2)^{r} \left| \widehat{\varphi_3}(k) \right|^2 \right)^{1/2} \\
        & \quad + \left(\frac{2\pi}{M^4} \sum_{k=-\infty}^{\infty}(1+k^2)^{s} \left| \widehat{\varphi_2}(k) \right|^2 \right)^{1/2} \\
        & \leq \left(2\pi\sum_{k=-\infty}^{\infty}(1+k^2)^{q-1} \left| \widehat{\varphi_1}(k) \right|^2 \right)^{1/2} \\
        & \quad + \left( |\delta|^2 2\pi\sum_{k=-\infty}^{\infty}(1+k^2)^{q} \left| \widehat{\varphi_1}(k) \right|^2 \right)^{1/2} \\
        & \quad + \left(2\pi\sum_{k=-\infty}^{\infty}(1+k^2)^{r} \left| \widehat{\varphi_3}(k) \right|^2 \right)^{1/2} \\
        & \quad + \left(\frac{2\pi}{M^4}  \sum_{k=-\infty}^{\infty}(1+k^2)^{s} \left| \widehat{\varphi_2}(k) \right|^2 \right)^{1/2} \\
        & = \norm{\varphi_1}_{q-1} + |\delta|\norm{\varphi_1}_{q} + \frac{1}{M^2} \norm{\varphi_2}_{s} + \norm{\varphi_3}_{r} \\
        &\leq \paren{1+|\delta|}\norm{\varphi_1}_{q} + \frac{1}{M^2} \norm{\varphi_2}_{s} + \norm{\varphi_3}_{r}\leq C_{\delta,M} \norm{\varphi}_{W},
    \end{split}
\end{equation}

where $C_{\delta,M}=\max \{1+|\delta|, \frac{1}{M^2}\}$. Since $A$ is a linear operator, \eqref{A} implies that 
\begin{equation}
    A \in \mathcal{B}\left(W, \tilde{W}\right).
\end{equation}
To simplify the notation, throughout this subsection let us consider $q,s \in \mathbb{R}$ arbitrary but fixed, and define
$$X:=H^q(\mathbb{T})\times H^s(\mathbb{T})\times H^{s+1}(\mathbb{T}), \quad  \tilde{X}:=H^{q-2}(\mathbb{T})\times H^{s-1}(\mathbb{T})\times H^s(\mathbb{T}).$$
In particular,
\begin{corolario}
     \begin{equation}
    A \in \mathcal{B}\left(X, \tilde{X}\right).
\end{equation}
\end{corolario}

On the other hand, since $\mathcal{A}(k)$ is diagonalizable for each $k$, the exponential of a diagonalizable matrix can be computed directly from its eigenvalues and eigenvectors, which gives
\begin{equation}
    e^{\mathcal{A}(k)t}=\left(\begin{array}{ccc}e^{it(\sigma_3k-\delta k^2)} & 0 & 0\\ 0 & \cos \left(\frac{k}{M} t\right) & M k \sin \left(\frac{k}{M} t\right) \\ 0 & -\frac{1}{M k}\sin \left(\frac{k}{M} t\right) & \cos \left(\frac{k}{M} t\right)\end{array}\right).
\end{equation}
That is, solving the ODE \eqref{EDO} with respect to $t$ we obtain $\widehat{\Psi(t)}(k)=e^{\mathcal{A}(k) t} \widehat{\Psi_0}(k)$, for each $k\in \mathbb{Z}$.
\\

Then, by the inverse Fourier transform, the candidate solution of \eqref{sun-jin} is formally 
\begin{equation}
    \Psi(t)=\sum_{k\in \mathbb{Z}}e^{ikx}e^{\mathcal{A}(k) t} \widehat{\Psi_0}(k).
\end{equation}
To prove that this candidate is indeed a solution, we define the family of operators $\{S(t)\}_{t\geq 0}$ as
\begin{equation}
S(t)\varphi=\left( \paren{e^{\mathcal{A}(k) t} \widehat{\varphi}(k)}_{k \in \mathbb{Z} }\right)^\lor, \quad \varphi \in \paren{\mathcal{P}'}^3,    
\end{equation}

and we have:
\begin{teorema}
    $\{S(t)\}_{t\geq 0}$ is a $C_0$ semigroup on $X$.
\end{teorema}
\begin{proof}
    Let us show that $S(t)\in \mathcal{B}(X)$ for every $t> 0$ (the case $t=0$ is trivial). Let $t>0$ be fixed, and let $\varphi=(\varphi_1, \varphi_2, \varphi_3)\in X$. It is easy to see that $S(t)$ is a linear operator and that
\begin{equation}
    \widehat{S(t)\varphi}(k)=\left(\begin{array}{c}
e^{it(\sigma_3 k -\delta k^2) } \widehat{\varphi_1}(k) \\
\cos \left(\frac{k}{M} t\right) \widehat{\varphi_2}(k)+ Mk \sin \left(\frac{k}{M} t\right) \widehat{\varphi_3}(k) \\
-\frac{1}{Mk}\sin \left(\frac{k}{M} t\right) \widehat{\varphi_2}(k)+ \cos \left(\frac{k}{M} t\right) \widehat{\varphi_3}(k)
\end{array}\right).
\end{equation}

Now, note that 
\begin{equation}
    \lim_{x\rightarrow 0}\frac{\sin \left(\frac{x}{M} t\right)}{Mx}=\lim_{u\rightarrow 0}\frac{\sin (u)}{M\left(\frac{uM}{t} \right)}=\frac{t}{M^2},
\end{equation}

so
\begin{equation}\label{t/M^2}
    \lim_{k\rightarrow 0}\left(\left|\frac{\sin \left(\frac{k}{M} t\right)}{Mk} \right|\sqrt{1+k^2}\right)=\frac{t}{M^2}.
\end{equation}

Moreover, the function $f(x)=\frac{\sqrt{1+x^2}}{\abs{x}}$ is symmetric about the $y$-axis ($f(x)=f(-x)$) and decreasing for $x>0$, since
\begin{equation}
    f^{\prime}(x)=\frac{\frac{x^2}{\sqrt{1+x^2}}-\sqrt{1+x^2}}{x^2}=-\frac{1}{x^2\sqrt{1+x^2}}<0, \quad \text{for all} \quad x>0.
\end{equation}
So
\begin{equation}\label{2/M}
    \begin{aligned}
    \sup_{k\in \mathbb{Z}\smallsetminus \{0\}}\left|\frac{\sin \left(\frac{k}{M} t\right)}{Mk} \right|\sqrt{1+k^2}&\leq \sup_{k\in \mathbb{Z}\smallsetminus \{0\}}\left|\frac{1}{Mk} \right|\sqrt{1+k^2}\\
    &=\sup_{k\in \mathbb{Z}^+}\left|\frac{1}{Mk} \right|\sqrt{1+k^2}\\
    &=\left.\left(\left|\frac{1}{Mk} \right|\sqrt{1+k^2}\right)\right|_{k=1}\\
    &=\frac{\sqrt{2}}{M}.
\end{aligned}
\end{equation}

Defining
\begin{equation}
    C_{M,t}:=\max\left\{\frac{t}{M^2}, \frac{\sqrt{2}}{M}\right\},
\end{equation}

it follows from \eqref{t/M^2} and \eqref{2/M} that
\begin{equation}
    \sup_{k\in \mathbb{Z}}\left|\frac{\sin \left(\frac{k}{M} t\right)}{Mk} \right|\sqrt{1+k^2}\leq C_{M,t} .
\end{equation}

Then
\begin{equation}\label{C_{M,t}}
    \begin{aligned}
    \left( 2\pi \sum (1+k^2)^{s+1}\left|\frac{\sin \left(\frac{k}{M} t\right)}{Mk} \widehat{\varphi_2}(k)\right|^2\right)^{1/2}&\leq \left(\sup_{k\in \mathbb{Z}}\left(\left|\frac{\sin \left(\frac{k}{M} t\right)}{Mk} \right|\sqrt{1+k^2}\right)^2 2\pi \sum (1+k^2)^{s}\left| \widehat{\varphi_2}(k)\right|^2\right)^{1/2}\\
    &\leq C_{M,t}\norm{\varphi_2}_s.
\end{aligned}
\end{equation}

Thus
\begin{equation}
    \norm{S(t)\varphi}_{X}\leq \norm{\varphi_1}_q+\norm{\varphi_2}_s+M\norm{\varphi_3}_{s+1}+C_{M,t}\norm{\varphi_2}_s+\norm{\varphi_3}_{s+1}\leq K_{M,t}\norm{\varphi}_{X},
\end{equation}
where $K_{M,t}=\max\left\{M, C_{M,t}\right\}+1$. This shows $S(t)\in \mathcal{B}(X)$.
\\

Moreover
\begin{itemize}
    \item $S(0)\varphi=\left(  \widehat{\varphi}(k)\right)^\lor=\varphi$, that is, $S(0)=I_X$.
    \item $S(t+s)=\mathcal{F}^{-1}e^{\mathcal{A}(k)(t+s)}\mathcal{F}=\mathcal{F}^{-1}e^{\mathcal{A}(k)t}e^{\mathcal{A}(k)s}\mathcal{F}=\mathcal{F}^{-1}e^{\mathcal{A}(k)t}\mathcal{F}\mathcal{F}^{-1}e^{\mathcal{A}(k)s}\mathcal{F}=S(t)S(s).$
\end{itemize}

That is, $\{S(t)\}_{t\geq 0}$ is a semigroup on $X$. To see that it is of class $C_0$, it remains to show that 
$$\lim_{t\rightarrow 0^+}\norm{S(t)\varphi-\varphi}_{X}=0.$$

It is clear that
\begin{equation}\label{series}
    \begin{aligned}
        \norm{S(t)\varphi-\varphi}_{X}&  =\left(2\pi\sum_{k=-\infty}^{\infty}(1+k^2)^{q}\left|e^{it(\sigma_3 k -\delta k^2) } \widehat{\varphi_1}(k)-\widehat{\varphi_1}(k)\right|^2\right)^{1/2}\\
        & +\left(2\pi\sum_{k=-\infty}^{\infty}(1+k^2)^{s}\left|\cos \left(\frac{k}{M} t\right)\widehat{\varphi_2}(k)-\widehat{\varphi_2}(k)+Mk\sin \left(\frac{k}{M} t\right)\widehat{\varphi_3}(k)\right|^2\right)^{1/2}\\
        & + \left(2\pi\sum_{k=-\infty}^{\infty}(1+k^2)^{s+1}\left|-\frac{1}{Mk}\sin \left(\frac{k}{M} t\right) \widehat{\varphi_2}(k)+\cos \left(\frac{k}{M} t\right)\widehat{\varphi_3}(k)-\widehat{\varphi_3}(k)\right|^2\right)^{1/2}.
    \end{aligned}
\end{equation}

For the first series, note that
\begin{equation}\label{w1}
  (1+k^2)^{q}\abs{e^{it(\sigma_3 k -\delta k^2) } \widehat{\varphi_1}(k)-\widehat{\varphi_1}(k)}^2 \leq (1+k^2)^{q}\abs{2\widehat{\varphi_1}(k)}^2=4(1+k^2)^{q}\abs{\widehat{\varphi_1}(k)}^2, \quad k\in \mathbb{Z},
\end{equation}

and since $\varphi_1 \in H^q(\mathbb{T})$,
$$\sum _{k=-\infty}^{\infty}(1+k^2)^{q}\abs{\widehat{\varphi_1}(k)}^2<\infty.$$
By the Weierstrass $M$-test, the series 
\begin{equation}
    \sum_{k=-\infty}^{\infty}(1+k^2)^{q}\left|e^{it(\sigma_3 k -\delta k^2) } \widehat{\varphi_1}(k)-\widehat{\varphi_1}(k)\right|^2
\end{equation}
converges uniformly for all $t\geq 0$. 
\\

For the second series, since $(a+b)^2\leq 2(a^2+b^2)$ for all $a, b \in \mathbb{R}$,

\begin{equation}\label{w2}
\begin{split}
    (1+k^2)^{s}&\left|\cos \left(\frac{k}{M} t\right)\widehat{\varphi_2}(k)-\widehat{\varphi_2}(k)+Mk\sin \left(\frac{k}{M} t\right)\widehat{\varphi_3}(k)\right|^2 \\
    &\leq 8(1+k^2)^{s}\abs{\widehat{\varphi_2}(k)}^2+2M^2(1+k^2)^{s+1}\abs{\widehat{\varphi_3}(k)}^2,
\end{split}
\end{equation}
and since $\varphi_2 \in H^s(\mathbb{T})$ and $\varphi_3 \in H^{s+1}(\mathbb{T})$, the right-hand side is $k$-summable. By the Weierstrass $M$-test again, the series
\begin{equation}
    \sum _{k=-\infty}^{\infty}(1+k^2)^{s}\left|\cos \left(\frac{k}{M} t\right)\widehat{\varphi_2}(k)-\widehat{\varphi_2}(k)+Mk\sin \left(\frac{k}{M} t\right)\widehat{\varphi_3}(k)\right|^2
\end{equation}
converges uniformly for all $t\geq0$.
\\

Similarly, using \eqref{C_{M,t}}, the series

\begin{equation}
\sum_{k=-\infty}^{\infty}(1+k^2)^{s+1}\left|-\frac{1}{Mk}\sin \left(\frac{k}{M} t\right) \widehat{\varphi_2}(k)+\cos \left(\frac{k}{M} t\right)\widehat{\varphi_3}(k)-\widehat{\varphi_3}(k)\right|^2
\end{equation}

converges uniformly for $t\geq0$. 
\\

Since the three series in \eqref{series} converge uniformly in $t$, uniform convergence lets us interchange the limit $t\to 0^+$ with the sum in each series. In every term, each summand tends to $0$ as $t\to 0^+$ (using $e^{i0}=1$, $\cos(0)=1$, $\sin(0)=0$), so
\begin{equation}
    \lim_{t\rightarrow 0^+} \norm{S(t)\varphi-\varphi}_{X} =0.
\end{equation}

This completes the proof.
\end{proof}
\medskip

Now $S(t)\Psi_0$ is the solution of the linear equation. More precisely:

\begin{teorema}\label{solución de 3.2} Let $\Psi_0\in X$. Then $\Psi(t)=S(t)\Psi_0$, $t\geq 0$, is the unique function satisfying
\begin{equation}
  \Psi(t) \in C\paren{[0,\infty);X}, \quad\lim _{h \rightarrow 0}\left\|\frac{\Psi(t+h)-\Psi(t)}{h}-A \Psi(t)\right\|_{\tilde{X}}=0, \quad \Psi (0)= \Psi_0 . 
\end{equation}
That is, $\Psi(t)$ is the unique solution of \eqref{sun-jin}. The above convergence is uniform with respect to $t\geq 0$.
\end{teorema}
\begin{proof}
Since $\{S(t)\}$ is a $C_0$ semigroup, $t\mapsto S(t)\Psi_0$ is a continuous function of $t$, so $\Psi(t) \in C\paren{[0,\infty);X}$. Now let $t\geq0, h>0$. It is clear that
$$
    \left({\frac{\Psi(t+h)-\Psi(t)}{h}}\right)^{\land}(k)  =\left({\frac{S(t+h)\Psi_0-S(t)\Psi_0}{h}}\right)^{\land}(k)$$
    $$=\frac{1}{h
    }\left(\begin{array}{c}
e^{i(t+h)(\sigma_3 k -\delta k^2) } \widehat{\psi_0}(k)-e^{it(\sigma_3 k -\delta k^2) } \widehat{\psi_0}(k) \\
\cos \left(\frac{k}{M} (t+h)\right) \widehat{\rho_0}(k)-\cos \left(\frac{k}{M} t\right) \widehat{\rho_0}(k)+ Mk \sin \left(\frac{k}{M} (t+h)\right) \widehat{\phi_0}(k)-Mk \sin \left(\frac{k}{M} t\right) \widehat{\phi_0}(k) \\
-\frac{1}{Mk}\sin \left(\frac{k}{M} (t+h)\right) \widehat{\rho_0}(k)+\frac{1}{Mk}\sin \left(\frac{k}{M} t\right) \widehat{\rho_0}(k)+ \cos \left(\frac{k}{M} (t+h)\right) \widehat{\phi_0}(k)-\cos \left(\frac{k}{M} t\right) \widehat{\phi_0}(k)
\end{array}\right),
$$

and

 $$\begin{aligned}
    \widehat{A\Psi(t)}(k) & =\widehat{AS(t)\Psi_0}(k)\\
    & = \left(\begin{array}{c}
(i\sigma_3k-i \delta k^2)e^{it(\sigma_3 k -\delta k^2) } \widehat{\psi_0}(k)  \\
-\frac{k}{M}\sin \left(\frac{k}{M} t\right) \widehat{\rho_0}(k)+ k^2\cos \left(\frac{k}{M} t\right) \widehat{\phi_0}(k)
\\
 -\frac{1}{M^2}\cos \left(\frac{k}{M} t\right) \widehat{\rho_0}(k)- \frac{k}{M} \sin \left(\frac{k}{M} t\right) \widehat{\phi_0}(k)
\end{array}\right).
 \end{aligned}$$
 
Writing $f^1_{k,x}=\cos \left(\frac{k}{M} x\right)$ and $f^2_{k,x}=\sin \left(\frac{k}{M} x\right)$, we get
$$
    \left({\frac{\Psi(t+h)-\Psi(t)}{h}}-A\Psi(t)\right)^{\land}(k) $$
    $$\large{=\left(\begin{array}{c}
\frac{e^{i(t+h)(\sigma_3 k -\delta k^2) } \widehat{\psi_0}(k)-e^{it(\sigma_3 k -\delta k^2) } \widehat{\psi_0}(k)}{h} -(i\sigma_3 k -i\delta k^2)e^{it(\sigma_3 k -\delta k^2)}\widehat{\psi_0}(k)\\
\frac{f^1_{k,t+h}\widehat{\rho_0}(k)-f^1_{k,t}\widehat{\rho_0}(k)+Mkf^2_{k,t+h}\widehat{\phi_0}(k)-Mkf^2_{k,t}\widehat{\phi_0}(k)}{h}-\left(-\frac{k}{M}f^2_{k,t}\widehat{\rho_0}(k)+k^2f^1_{k,t}\widehat{\phi_0}(k)\right) \\
\frac{-\frac{1}{Mk}f^2_{k,t+h} \widehat{\rho_0}(k)+\frac{1}{Mk}f^2_{k,t} \widehat{\rho_0}(k)+ f^1_{k,t+h} \widehat{\phi_0}(k)-f^1_{k,t} \widehat{\phi_0}(k)}{h}-\left(-\frac{1}{M^2}f^1_{k,t} \widehat{\rho_0}(k)- \frac{k}{M} f^2_{k,t} \widehat{\phi_0}(k)\right)
\end{array}\right).}
$$

Therefore
$$\left\|\frac{\Psi(t+h)-\Psi(t)}{h}-A \Psi(t)\right\|_{\tilde{X}}$$
$$=\left(2\pi\sum_{k=-\infty}^{\infty}(1+k^2)^{q-2}\left|\frac{e^{i(t+h)(\sigma_3 k -\delta k^2) }-e^{it(\sigma_3 k -\delta k^2) }}{h}-(i\sigma_3 k -i\delta k^2)e^{it(\sigma_3 k -\delta k^2) }\right|^2|\widehat{\psi_0}(k)|^2\right)^{1/2}$$
$$+\left(2\pi\sum_{k=-\infty}^{\infty}(1+k^2)^{s-1}\left|\left(\frac{f^1_{k,t+h}-f^1_{k,t}}{h}+\frac{k}{M}f^2_{k,t}\right)\widehat{\rho_0}(k)+\left(\frac{Mk(f^2_{k,t+h}-f^2_{k,t})}{h}-k^2f^1_{k,t}\right)\widehat{\phi_0}(k)\right|^2\right)^{1/2}$$
\begin{equation}\label{3.29}
    +\left(2\pi\sum_{k=-\infty}^{\infty}(1+k^2)^{s}\left|\left(\frac{-\frac{1}{Mk}(f^2_{k,t+h}-f^2_{k,t})}{h}+\frac{1}{M^2}f^1_{k,t}\right)\widehat{\rho_0}(k)+\left(\frac{f^1_{k,t+h}-f^1_{k,t}}{h}+\frac{k}{M}f^2_{k,t}\right)\widehat{\phi_0}(k)\right|^2\right)^{1/2}.
\end{equation}

We now check that the three series in \eqref{3.29} converge uniformly in $h$. 
\\

Since  $e^{ix(\sigma_3 k -\delta k^2) }\in C^1([0,\infty))$, the Mean Value Theorem gives some $a\in (0,h)$ such that

$$\left|\frac{e^{ih(\sigma_3 k -\delta k^2) }-1}{h}\right|=\left|(i\sigma_3 k -i\delta k^2)e^{ia(\sigma_3 k -\delta k^2) }\right|=|\sigma_3 k -\delta k^2|\leq |k|+|\delta||k|^2 \leq 2C_{\delta}k^2,$$

where $C_{\delta}:=\max\{1,|\delta|\}$. Then

\begin{equation}\label{3.30}
    \begin{aligned}
    2\pi&\sum_{k=-\infty}^{\infty}(1+k^2)^{q-2}\left|\frac{e^{i(t+h)(\sigma_3 k -\delta k^2) }-e^{it(\sigma_3 k -\delta k^2) }}{h}-(i\sigma_3 k -i\delta k^2)e^{it(\sigma_3 k -\delta k^2) }\right|^2|\widehat{\psi_0}(k)|^2\\
    &=2\pi\sum_{k=-\infty}^{\infty}(1+k^2)^{q-2}\left|e^{it(\sigma_3 k -\delta k^2) }\right|^2\left|\frac{e^{ih(\sigma_3 k -\delta k^2) }-1}{h}-(i\sigma_3 k -i\delta k^2)\right|^2|\widehat{\psi_0}(k)|^2\\
    &\leq2\pi\sum_{k=-\infty}^{\infty}(1+k^2)^{q-2}\left|4C_{\delta}k^2\right|^2|\widehat{\psi_0}(k)|^2\\
    &\leq32\pi C_{\delta}^2\sum_{k=-\infty}^{\infty}(1+k^2)^{q}|\widehat{\psi_0}(k)|^2<\infty.
\end{aligned}
\end{equation}

The sum formulas for sine and cosine give

\begin{equation*}
    f^1_{k,t+h}=f^1_{k,t}f^1_{k,h}-f^2_{k,t}f^2_{k,h}, \qquad
    f^2_{k,t+h}=f^2_{k,t}f^1_{k,h}+f^2_{k,h}f^1_{k,t}.
\end{equation*}

Applying the Mean Value Theorem again, there exist $b, c\in (0,h)$ such that

$$\left|\frac{f^1_{k,h}-1}{h}\right|=\left|-\frac{k}{M}\sin \left(\frac{k}{M} b\right)\right|\leq\left|\frac{k}{M}\right|, \qquad
\left|\frac{f^2_{k,h}}{h}\right|=\left|\frac{k}{M}\cos \left(\frac{k}{M} c\right)\right|\leq\left|\frac{k}{M}\right|.$$

Using these two bounds, we obtain
 \begin{equation}\label{3.31}
 \begin{aligned}
     2\pi&\sum_{k=-\infty}^{\infty}(1+k^2)^{s-1}\left|\left(\frac{f^1_{k,t+h}-f^1_{k,t}}{h}+\frac{k}{M}f^2_{k,t}\right)\widehat{\rho_0}(k)\right|^2\\
     &=2\pi\sum_{k=-\infty}^{\infty}(1+k^2)^{s-1}\left|f^1_{k,t}\left(\frac{f^1_{k,h}-1}{h}\right)-f^2_{k,t}\left(\frac{f^2_{k,h}}{h}\right)+\frac{k}{M}f^2_{k,t}\right|^2|\widehat{\rho_0}(k)|^2\\
     &\leq2\pi\sum_{k=-\infty}^{\infty}(1+k^2)^{s-1}\left|\frac{3k}{M}\right|^2|\widehat{\rho_0}(k)|^2\leq\frac{18\pi}{M^2}\sum_{k=-\infty}^{\infty}(1+k^2)^{s}|\widehat{\rho_0}(k)|^2<\infty,
 \end{aligned}
 \end{equation}
and, by the same argument applied to the second component of $\widehat{\phi_0}(k)$,
\begin{equation}\label{3.32}
\begin{aligned}
    2\pi&\sum_{k=-\infty}^{\infty}(1+k^2)^{s-1}\left|\left(\frac{Mk(f^2_{k,t+h}-f^2_{k,t})}{h}-k^2f^1_{k,t}\right)\widehat{\phi_0}(k)\right|^2\\
    &=2\pi\sum_{k=-\infty}^{\infty}(1+k^2)^{s-1}\left|Mkf^2_{k,t}\left(\frac{f^1_{k,h}-1}{h}\right)+Mkf^1_{k,t}\left(\frac{f^2_{k,h}}{h}\right)-k^2f^1_{k,t}\right|^2|\widehat{\phi_0}(k)|^2\\
    & \leq2\pi\sum_{k=-\infty}^{\infty}(1+k^2)^{s-1}\left|3k^2\right|^2|\widehat{\phi_0}(k)|^2\leq18\pi\sum_{k=-\infty}^{\infty}(1+k^2)^{s+1}|\widehat{\phi_0}(k)|^2<\infty.
\end{aligned}
\end{equation}

The two remaining terms in \eqref{3.29} (involving $\widehat{\rho_0}(k)$ and $\widehat{\phi_0}(k)$ in the third component) are bounded by exactly the same method, using the same two Mean Value Theorem estimates above; the resulting bounds are
\begin{equation}\label{3.33}
    2\pi\sum_{k=-\infty}^{\infty}(1+k^2)^{s}\left|\left(\frac{-\frac{1}{Mk}(f^2_{k,t+h}-f^2_{k,t})}{h}+\frac{1}{M^2}f^1_{k,t}\right)\widehat{\rho_0}(k)\right|^2\leq\frac{18\pi}{M^4}\sum_{k=-\infty}^{\infty}(1+k^2)^{s}|\widehat{\rho_0}(k)|^2<\infty,
\end{equation}
\begin{equation}\label{3.34}
    2\pi\sum_{k=-\infty}^{\infty}(1+k^2)^{s}\left|\left(\frac{f^1_{k,t+h}-f^1_{k,t}}{h}+\frac{k}{M}f^2_{k,t}\right)\widehat{\phi_0}(k)\right|^2\leq\frac{18\pi}{M^2}\sum_{k=-\infty}^{\infty}(1+k^2)^{s+1}|\widehat{\phi_0}(k)|^2<\infty.
\end{equation}

With \eqref{3.30}, \eqref{3.31}, \eqref{3.32}, \eqref{3.33} and \eqref{3.34}, we have shown that the three series in \eqref{3.29} converge uniformly for $h>0$, by the Weierstrass $M$-test. This uniform convergence lets us take the limit $h\to 0^+$ term by term inside each series. In every term, the difference quotient tends to the corresponding derivative as $h\to0^+$, so each term cancels exactly with the corresponding part coming from $\widehat{A\Psi(t)}(k)$, giving
\begin{equation}\label{3.35}
    \lim _{h \rightarrow 0^+}\left\|\frac{\Psi(t+h)-\Psi(t)}{h}-A \Psi(t)\right\|_{\tilde{X}}=0.
\end{equation}

Let us now consider the case $0<-h<t$. Using the semigroup property $S(t+h)=S(t)S(h)$ and the boundedness of $S(t+h)$ established above,
\begin{equation}
\begin{aligned}
         \left\|\frac{\Psi(t+h)-\Psi(t)}{h}-A \Psi(t)\right\|_{\tilde{X}}  &=     \left\|\frac{S(t+h)\Psi_0-S(t)\Psi_0}{h}-A S(t)\Psi_0\right\|_{\tilde{X}}  \\
         &=\left\|S(t+h)\left(\frac{S(-h)\Psi_0-\Psi_0}{-h}-A S(-h)\Psi_0\right)\right\|_{\tilde{X}}\\
         &\leq K_{M,t} \left\|\frac{S(-h)\Psi_0-\Psi_0}{-h}-A\Psi_0\right\|_{\tilde{X}}+K_{M,t} C_{\delta,M}\left\|\Psi_0- S(-h)\Psi_0\right\|_{X}.
\end{aligned}
\end{equation}

By what was shown for $h>0$ and the strong continuity of the semigroup, both terms on the right tend to $0$ as $h\to0^-$, so
\begin{equation}\label{3.37}
    \lim _{h \rightarrow 0^-}\left\|\frac{\Psi(t+h)-\Psi(t)}{h}-A \Psi(t)\right\|_{\tilde{X}}=0.
\end{equation}
With \eqref{3.35} and \eqref{3.37} we have shown that, for every $t\geq0$, $\partial_t\Psi(t)=A\Psi (t)$ in the space $\tilde{X}$. Moreover $\Psi (0)=S(0)\Psi_0=\Psi_0$, so $\Psi(t)$ is a solution of the IVP \eqref{sun-jin}.
\\

For uniqueness, suppose $u(t) \in C\paren{[0,\infty);X}$ is a solution of \eqref{sun-jin}. Recall that the Fourier transform of a periodic distribution $f$ is defined by $\widehat{f}(k)=\frac{1}{2\pi}\langle f, e^{-ik(\cdot)}\rangle$, so
\begin{equation}
    \partial_t\widehat{u(t)}(k) = \frac{1}{2\pi}\lim _{h \rightarrow 0}\left\langle \frac{u(t+h)-u(t)}{h}, e^{-ik(\cdot)}\right\rangle, \quad k \in \mathbb{Z}.
\end{equation}
Since the strong convergence
$$\lim _{h \rightarrow 0}\left\|\frac{u(t+h)-u(t)}{h}-A u(t)\right\|_{\tilde{X}}=0$$
implies convergence when tested against any fixed $\varphi\in\mathcal P$ (in particular against $\varphi=e^{-ik(\cdot)}$), we get
\begin{equation}
    \partial_t\widehat{u(t)}(k)=\frac{1}{2\pi}\langle Au(t), e^{-ik(\cdot)}\rangle=\mathcal{A}(k)\widehat{u(t)}(k),
\end{equation}
so $u(t)$ must satisfy the decoupled ODE system
\begin{equation}\label{EDOu}
    \left\{\begin{array}{l}
\partial_t\widehat{u(t)}(k)=\mathcal{A}(k)\widehat{u(t)}(k), \quad k\in \mathbb{Z}, t\geq 0 ,\\
\widehat{u(0)}(k)=\widehat{\Psi_0}(k) .
\end{array}\right.
\end{equation} 
whose unique solution is $\widehat{u(t)}(k)=e^{\mathcal{A}(k) t} \widehat{\Psi_0}(k)$, for each $k\in \mathbb{Z}$. Thus
\begin{equation}
    u(t)=\sum_{k\in \mathbb{Z}}e^{ikx}e^{\mathcal{A}(k) t} \widehat{\Psi_0}(k)=S(t)\Psi_0=\Psi(t).
\end{equation}
\end{proof}
\medskip
\textbf{\textit{Proof of Theorem \ref{Teorema 3.1}.}} By Theorem \ref{solución de 3.2}, it remains to check that $\Psi(t) \in C^1\left([0,\infty);\tilde{X}\right)$ and that the solution depends continuously on the initial data.
\\

Let $t\geq 0$. Again by Theorem \ref{solución de 3.2},
\begin{equation}
\begin{aligned}
        \left\|\partial_t\Psi(t+h)-\partial_t\Psi(t)\right\|_{\tilde{X}}&=\left\|A(S(t+h)\Psi_o-S(t)\Psi_o)\right\|_{\tilde{X}}\leq C_{\delta,M}\left\|S(t+h)\Psi_o-S(t)\Psi_o\right\|_{X}.
\end{aligned}
\end{equation}
By continuity of $S(t)$, the right-hand side tends to $0$ as $h\to0$, so $\Psi(t)$ is continuously differentiable.
\\

On the other hand, let $\Psi_{0,1},\Psi_{0,2} \in X$, and let $\Psi_1(t), \Psi_2(t) \in  C\paren{[0,\infty);X}$ be the associated solutions of \eqref{sun-jin}. Then
\begin{equation}
    \sup_{t\in [0,\infty)}\norm{\Psi_1(t)-\Psi_2(t)}_X=\sup_{t\in [0,\infty)}\norm{S(t)\Psi_{0,1}-S(t)\Psi_{0,2}}_X\leq K_{M,t}\norm{\Psi_{0,1}-\Psi_{0,2}}_X .
\end{equation}
\hfill\qedsymbol

\subsection{Modified Zakharov-Rubenchik System}\label{subseccion ZR modificado}

As explained in the derivation of the Benney-Roskes system from the Euler equations with free surface (following Obrecht \cite{obrecht2015approximation}), the second and third equations of the Zakharov-Rubenchik system may be modified by adding $O(\epsilon)$ terms without changing the order of consistency of the system. We will therefore consider the following modified system:
\begin{equation}\label{ZR modificado}
    \left\{\begin{array}{l}
\partial_t\psi-\sigma_3 \partial_x\psi-i\epsilon \delta \partial_x^2\psi+i\epsilon\left\{\sigma_2|\psi|^2+W\left(\rho+D \partial_x\phi\right)\right\} \psi=0, \\
\partial_t\rho+\partial_x^2\phi+D\partial_x\left(|\psi|^2\right)+\epsilon L_1=0, \\
\partial_t\phi+\frac{1}{M^2} \rho+|\psi|^2+\epsilon L_2=0,
\end{array}\right.
\tag{mZR}
\end{equation}

with initial conditions $\psi (t=0)=\psi_0\in H^s(\mathbb{T})$, $\rho(t=0)=\rho_0\in H^s(\mathbb{T})$, $\phi(t=0)=\phi_0\in H^{s+1}(\mathbb{T})$, $s>1/2$, and we will choose $L_1,L_2$ conveniently.

\subsection{Auxiliary System Derived from the Modified Zakharov-Rubenchik System}\label{deduccion de sistema auxiliar}
In this section the variables $\rho, \phi$ will be replaced by $U, V$, and the values of $L_1, L_2$ will be chosen so that the resulting second and third equations are linear. Local well-posedness will be shown for this new system.
\\

Multiplying the first equation of \eqref{ZR modificado} by $\overline{\psi}$ and taking the real part, we have
$$Re\{\partial_t\psi\overline{\psi}\}-Re\{\sigma_3 \partial_x\psi\overline{\psi}\}-Re\{i\epsilon \delta \partial_x^2\psi\overline{\psi}\}+Re\{i\epsilon\left\{\sigma_2|\psi|^2+W\left(\rho+D \partial_x\phi\right)\right\} \psi\overline{\psi}\}=0.$$

Since $Re\{z\}=\frac{1}{2}\left(z+\overline{z}\right)$ for all $z\in \mathbb{C}$, and $Re\{i\epsilon\left\{\sigma_2|\psi|^2+W\left(\rho+D \partial_x\phi\right)\right\} |\psi|^2\}=0$, then
$$\begin{aligned}
    \frac{1}{2}\left(i\epsilon \delta \partial_x^2\psi\overline{\psi}+(-i)\epsilon \delta \partial_x^2\overline{\psi}\psi\right)&=\frac{1}{2}\left(\partial_t\psi\overline{\psi}+\partial_t\overline{\psi}\psi\right)-\frac{1}{2}\left(\sigma_3 \partial_x\psi\overline{\psi}+\sigma_3 \partial_x\overline{\psi}\psi\right)=\frac{1}{2}\partial_t\left(|\psi|^2\right)-\frac{\sigma_3}{2}\partial_x\left( |\psi|^2\right).
\end{aligned}$$

That is,
\begin{equation}\label{igualdad para simplificar}
    \partial_t\left(|\psi|^2\right)-\sigma_3\partial_x\left(|\psi|^2\right)=i\epsilon \delta \partial_x^2\psi\overline{\psi}-i\epsilon \delta \partial_x^2\overline{\psi}\psi.
\end{equation}

Now set $Q=M\partial_x\phi$ and differentiate the third equation of \eqref{ZR modificado} with respect to $x$:
\begin{equation}\label{4.17}
    \left\{\begin{array}{l}
\partial_t\rho+\frac{1}{M}\partial_xQ+D\partial_x\left(|\psi|^2\right)+\epsilon L_1=0, \\
\partial_tQ+\frac{1}{M} \partial_x\rho+M\partial_x\left(|\psi|^2\right)+\epsilon M \partial_xL_2=0.
\end{array}\right.
\end{equation}

Following Obrecht \cite{obrecht2015approximation}, we perform the change of variables
\begin{equation}
    \rho=U+\alpha_1|\psi|^2, \hspace{1cm} Q=V+\alpha_2|\psi|^2. 
\end{equation}
Substituting into \eqref{4.17}:
\begin{equation}\label{sistema U, V}
     \left\{\begin{array}{l}
\partial_tU+\frac{1}{M}\partial_xV+\alpha_1\partial_t\left(|\psi|^2\right)+\left(\frac{1}{M}\alpha_2+D\right)\partial_x\left(|\psi|^2\right)+\epsilon L_1=0, \\
\partial_tV+\frac{1}{M} \partial_xU+\alpha_2\partial_t\left(|\psi|^2\right)+\left(\frac{1}{M}\alpha_1+M\right)\partial_x\left(|\psi|^2\right)+\epsilon M \partial_xL_2=0.
\end{array}\right.
\end{equation}

Next, we choose $\alpha_1,\alpha_2$ so that
\begin{equation}
    \left\{\begin{array}{l}
 \frac{1}{M}\alpha_2+D=-\sigma_3\alpha_1,\\
\frac{1}{M}\alpha_1+M=-\sigma_3\alpha_2.
\end{array}\right.
\end{equation}

This system has a unique solution if and only if $(\sigma_3)^2-\frac{1}{M^2}\neq0$. Since $\sigma_3 \in \{1,-1\}$, this holds provided $M\neq 1$. With this choice, \eqref{sistema U, V} becomes 
\begin{equation}
     \left\{\begin{array}{l}
\partial_tU+\frac{1}{M}\partial_xV+\alpha_1\left(\partial_t\left(|\psi|^2\right)-\sigma_3\partial_x\left(|\psi|^2\right)\right)+\epsilon L_1=0 ,\\
\partial_tV+\frac{1}{M} \partial_xU+\alpha_2\left(\partial_t\left(|\psi|^2\right)-\sigma_3\partial_x\left(|\psi|^2\right)\right)+\epsilon M \partial_xL_2=0,
\end{array}\right.
\end{equation}
and, using \eqref{igualdad para simplificar},
\begin{equation}
     \left\{\begin{array}{l}
\partial_tU+\frac{1}{M}\partial_xV+i\epsilon\delta\alpha_1\left( \partial_x^2\psi\overline{\psi}- \partial_x^2\overline{\psi}\psi\right)+\epsilon L_1=0 ,\\
\partial_tV+\frac{1}{M} \partial_xU+i\epsilon\delta\alpha_2\left( \partial_x^2\psi\overline{\psi}- \partial_x^2\overline{\psi}\psi\right)+\epsilon M \partial_xL_2=0.
\end{array}\right.
\end{equation}

So, choosing
$$L_1=-i\delta\alpha_1\left( \partial_x^2\psi\overline{\psi}- \partial_x^2\overline{\psi}\psi\right), \quad L_2=-\frac{i}{M}\delta\alpha_2\left( \partial_x\psi\overline{\psi}- \partial_x\overline{\psi}\psi\right),$$

and writing the first equation of \eqref{ZR modificado} in terms of $\psi$, $U$, $V$, the modified Zakharov-Rubenchik system becomes
\begin{equation}\label{sistema nuevo}
    \left\{\begin{array}{l}
\partial_t\Psi=A\Psi+B(\Psi) , \\
\Psi(t=0)=\Psi_0, 
\end{array}\right.
\end{equation}
where
$$
\Psi=\left(\begin{array}{c}
\psi \\
U \\
V
\end{array}\right),\quad A=\left(\begin{array}{ccc}
\sigma_3\partial_x+i \epsilon\delta\partial_x^2 & 0 & 0 \\
0 & 0 & -\frac{1}{M}\partial_x \\
0 & -\frac{1}{M}\partial_x  & 0
\end{array}\right), \quad 
\Psi_0=\left(\begin{array}{c}
\psi_0 \\
\rho_0-\alpha_1|\psi_0|^2 \\
M\partial_x\phi_0 -\alpha_2|\psi_0|^2
\end{array}\right),
$$
\begin{equation}\label{definición de Psi, A, B(Psi)}
    B(\Psi)=\left(\begin{array}{c}
-i\epsilon\left\{\left(\sigma_2+W\left(\alpha_1+\frac{D}{M}\alpha_2\right)\right)|\psi|^2+W\left(U+\frac{D}{M}V\right)\right\} \psi \\
0 \\
0
\end{array}\right).
\end{equation}

Since $s>1/2$, $H^s(\mathbb{T})$ is a Banach algebra and there is a constant $C_s>0$, depending only on $s$, such that $\|fg\|_s\le C_s\|f\|_s\|g\|_s$ for all $f,g\in H^s(\mathbb{T})$. Using this, we can see that $\Psi_0 \in Y$, where $Y$ denotes the space $\left(H^s(\mathbb{T})\right)^3$:
\begin{equation}\label{Psi0 in Y}
\begin{split}
    \norm{\Psi_0}_Y&=\norm{\psi_0}_s+\norm{\rho_0-\alpha_1|\psi_0|^2}_s+\norm{M\partial_x\phi_0 -\alpha_2|\psi_0|^2}_s\\
    & \leq \norm{\psi_0}_s+\norm{\rho_0}_s+\abs{\alpha_1}\norm{|\psi_0|^2}_s+M\norm{\partial_x\phi_0}_s+\abs{\alpha_2}\norm{|\psi_0|^2}_s\\
    & \leq \norm{\psi_0}_s+\norm{\rho_0}_s+\abs{\alpha_1}C_s\norm{\psi_0}^2_s+M\norm{\phi_0}_{s+1}+\abs{\alpha_2}C_s\norm{\psi_0}^2_s< \infty.
\end{split}
\end{equation}

\subsubsection{Linear Equation}
Consider the linear part of the system \eqref{sistema nuevo}:
\begin{equation}\label{sistema nuevo lineal}
    \left\{\begin{array}{l}
\partial_t\Psi=A\Psi , \\
\Psi(t=0)=\Psi_0 \in Y.
\end{array}\right.
\end{equation}

We proceed exactly as with the system \eqref{sun-jin}. Taking the Fourier transform gives
$$\widehat{A\varphi}(k)=\mathcal{A}(k)\widehat{\varphi}(k),$$
for every $k \in \mathbb{Z}$ and every $\varphi=(\varphi_1,\varphi_2,\varphi_3) \in \mathcal{P}^{\prime}\times\mathcal{P}^{\prime}\times\mathcal{P}^{\prime}$, where
$$\mathcal{A}(k)=\left(\begin{array}{ccc}
\sigma_3ik-i \epsilon\delta k^2 & 0 & 0 \\
0 & 0 & -i\frac{k}{M} \\
0 & -i\frac{k}{M}  & 0
\end{array}\right).$$

Setting $\tilde{Y}:=H^{s-2}(\mathbb{T})\times H^{s-1}(\mathbb{T})\times H^{s-1}(\mathbb{T})$, the same type of estimate used in \eqref{A} above gives

\begin{equation}
        \norm{A\varphi}_{\tilde{Y}} \leq \paren{1+\epsilon|\delta|}\norm{\varphi_1}_{s} + \frac{1}{M}\norm{\varphi_2}_{s} + \frac{1}{M}\norm{\varphi_3}_{s}\leq C_{\epsilon, \delta,M}\norm{\varphi}_{Y},
\end{equation}
where $C_{\epsilon, \delta,M}:=\max\{1+\epsilon|\delta|, \frac{1}{M}\}$. Since $A$ is linear, this shows $A \in \mathcal{B}\left(Y, \tilde{Y}\right)$.
\\

Taking the Fourier transform of \eqref{sistema nuevo lineal} gives, for each $k\in \mathbb{Z}$,
\begin{equation}\label{sistema nuevo lineal TF}
    \left\{\begin{array}{l}
\partial_t\widehat{\Psi(t)}(k)=\mathcal{A}(k)\widehat{\Psi(t)}(k), \quad t\geq 0, \\
\widehat{\Psi(0)}(k)=\widehat{\Psi_0}(k), 
\end{array}\right.
\end{equation}

a decoupled system of ODEs. Diagonalizing $\mathcal{A}(k)$,
$$e^{\mathcal{A}(k) t}=\left(\begin{array}{ccc}e^{it(\sigma_3k-\epsilon\delta k^2)} & 0 & 0\\ 0 & \cos \left(\frac{k}{M} t\right) & -i \sin \left(\frac{k}{M} t\right) \\ 0 & -i\sin \left(\frac{k}{M} t\right) & \cos \left(\frac{k}{M} t\right)\end{array}\right).$$  

Thus the unique solution of \eqref{sistema nuevo lineal TF} is $\widehat{\Psi(t)}(k)=e^{\mathcal{A}(k) t} \widehat{\Psi_0}(k)$, and by the inverse Fourier transform the candidate solution of \eqref{sistema nuevo lineal} is
\begin{equation}\label{solución de sistema nuevo lineal}
    \Psi(t)=\sum_{k=-\infty}^{\infty} e^{ikx}e^{\mathcal{A}(k) t} \widehat{\Psi_0}(k).
\end{equation}

Defining $S(t)\varphi=\left(\left( e^{\mathcal{A}(k) t} \widehat{\varphi}(k)\right)_{k\in \mathbb{Z}}\right)^\lor$ for $\varphi=(\varphi_1,\varphi_2,\varphi_3) \in \paren{\mathcal{P}^{\prime}}^3$, we have:

\begin{teorema}
    $\{S(t)\}_{t\geq 0}$ is a $C_0$ semigroup on $H^{q}(\mathbb{T})\times H^{s}(\mathbb{T})\times H^{s}(\mathbb{T})$, for every $q, s \in \mathbb{R}$.
\end{teorema}
\begin{proof}
    Let $q, s \in \mathbb{R}$, $\mathcal{H}:=H^{q}(\mathbb{T})\times H^{s}(\mathbb{T})\times H^{s}(\mathbb{T})$, and let $\varphi=(\varphi_1,\varphi_2,\varphi_3)\in \mathcal{H}$. We have
$$\widehat{S(t)\varphi}(k)=\left(\begin{array}{c}
e^{it(\sigma_3 k -\epsilon\delta k^2) } \widehat{\varphi_1}(k) \\
\cos \left(\frac{k}{M} t\right) \widehat{\varphi_2}(k)-i \sin \left(\frac{k}{M} t\right) \widehat{\varphi_3}(k) \\
-i\sin \left(\frac{k}{M} t\right) \widehat{\varphi_2}(k)+ \cos \left(\frac{k}{M} t\right) \widehat{\varphi_3}(k)\end{array}\right), \quad \forall k\in \mathbb{Z}.$$

Since every entry of the matrix above has modulus at most $1$, we get directly
\begin{equation}
    \norm{S(t)\varphi}_{\mathcal{H}}\leq \norm{\psi}_q+2\norm{\varphi_2}_s+2\norm{\varphi_3}_s\leq 2\norm{\varphi}_{\mathcal{H}},
\end{equation}
so $S(t)\in \mathcal{B}(\mathcal{H})$. As before, $S(0)=I_{\mathcal{H}}$ and $S(t+s)=S(t)S(s)$, so $\{S(t)\}_{t\geq 0}$ is a semigroup. The proof that the semigroup is of class $C_0$ is similar to the previous case.
\end{proof}

\begin{corolario}\label{$S(t)$ es un semigrupo de clase $C_0$ en $Y$}
    $S(t)$ is a $C_0$ semigroup on $Y$.
\end{corolario}

We now show that $S(t)\Psi_0$ is indeed the solution of \eqref{sistema nuevo lineal}.

\begin{teorema}\label{S(t)Psi_0 es la solución del sistema nuevo lineal}
   Let $\Psi_0 \in Y$. The function $\Psi(t)=S(t)\Psi_0$, $t\geq 0$, is the unique function satisfying
   \begin{equation}
       \Psi(t) \in C\paren{[0,\infty);Y}, \quad \lim _{h \rightarrow 0}\left\|\frac{\Psi(t+h)-\Psi(t)}{h}-A \Psi(t)\right\|_{\tilde{Y}}=0, \quad \Psi(0)=\Psi_0,
   \end{equation}
the convergence being uniform in $t$. 
\end{teorema}
 \begin{proof}
     By the previous corollary, $\Psi(t) \in C\paren{[0,\infty);Y}$. 
\\

Let $t\geq0$, $h>0$, and again write $f^1_{k,t}=\cos \left(\frac{k}{M} t\right)$, $f^2_{k,t}=\sin \left(\frac{k}{M} t\right)$. As in the proof of Theorem \ref{solución de 3.2},
$$\begin{aligned}
    \left({\frac{\Psi(t+h)-\Psi(t)}{h}}\right)^{\land}(k)&=\frac{1}{h
    }\left(\begin{array}{c}
e^{i(t+h)(\sigma_3 k -\epsilon\delta k^2) } \widehat{\psi_0}(k)-e^{it(\sigma_3 k -\epsilon\delta k^2) } \widehat{\psi_0}(k) \\
f^1_{k,t+h} \widehat{U_0}(k)-f^1_{k,t} \widehat{U_0}(k)-i f^2_{k,t+h} \widehat{V_0}(k)+i f^2_{k,t} \widehat{V_0}(k) \\
-if^2_{k,t+h} \widehat{U_0}(k)+if^2_{k,t} \widehat{U_0}(k)+ f^1_{k,t+h} \widehat{V_0}(k)-f^1_{k,t} \widehat{V_0}(k)
\end{array}\right),
\end{aligned}
$$

and
 $$\begin{aligned}
    \widehat{A\Psi(t)}(k) & = \left(\begin{array}{c}
(i\sigma_3k-i \epsilon\delta k^2)e^{it(\sigma_3 k -\epsilon\delta k^2) } \widehat{\psi_0}(k)  \\
-\frac{k}{M}f^2_{k,t} \widehat{U_0}(k)-i\frac{k}{M}f^1_{k,t} \widehat{V_0}(k)
\\
 -i\frac{k}{M}f^1_{k,t} \widehat{U_0}(k)- \frac{k}{M} f^2_{k,t} \widehat{V_0}(k)
\end{array}\right).
 \end{aligned}$$
 
Subtracting entry by entry and taking the $\tilde Y$-norm gives
$$\left\|\frac{\Psi(t+h)-\Psi(t)}{h}-A \Psi(t)\right\|_{\tilde{Y}}$$
    $$=\left(2\pi\sum_{k=-\infty}^{\infty}(1+k^2)^{s-2}\left|\frac{e^{i(t+h)(\sigma_3 k -\epsilon\delta k^2) }-e^{it(\sigma_3 k -\epsilon\delta k^2) }}{h}-(i\sigma_3 k -i\epsilon\delta k^2)e^{it(\sigma_3 k -\epsilon\delta k^2) }\right|^2|\widehat{\psi_0}(k)|^2\right)^{1/2}$$    
$$+\left(2\pi\sum_{k=-\infty}^{\infty}(1+k^2)^{s-1}\left|\left(\frac{f^1_{k,t+h}-f^1_{k,t}}{h}+\frac{k}{M}f^2_{k,t}\right)\widehat{U_0}(k)-i\left(\frac{f^2_{k,t+h}-f^2_{k,t}}{h}-\frac{k}{M}f^1_{k,t}\right)\widehat{V_0}(k)\right|^2\right)^{1/2}$$
\begin{equation}\label{tres series 2}
    +\left(2\pi\sum_{k=-\infty}^{\infty}(1+k^2)^{s-1}\left|-i\left(\frac{f^2_{k,t+h}-f^2_{k,t}}{h}-\frac{k}{M}f^1_{k,t}\right)\widehat{U_0}(k)+\left(\frac{f^1_{k,t+h}-f^1_{k,t}}{h}+\frac{k}{M}f^2_{k,t}\right)\widehat{V_0}(k)\right|^2\right)^{1/2}.
\end{equation}

The first series is bounded exactly as in \eqref{3.30}, using the Mean Value Theorem for $e^{ix(\sigma_3k-\epsilon\delta k^2)}$: there is $a\in(0,h)$ with $\left|\sigma_3 k-\epsilon\delta k^2\right|\le 2C_{\epsilon,\delta}k^2$, $C_{\epsilon,\delta}:=\max\{1,\epsilon|\delta|\}$, giving

$$2\pi\sum_{k=-\infty}^{\infty}(1+k^2)^{s-2}\left|\frac{e^{i(t+h)(\sigma_3 k -\epsilon\delta k^2) }-e^{it(\sigma_3 k -\epsilon\delta k^2) }}{h}-(i\sigma_3 k -i\epsilon\delta k^2)e^{it(\sigma_3 k -\epsilon\delta k^2) }\right|^2|\widehat{\psi_0}(k)|^2$$
    \begin{equation}\label{3.61}
        \leq32\pi C_{\epsilon,\delta}^2\sum_{k=-\infty}^{\infty}(1+k^2)^{s}|\widehat{\psi_0}(k)|^2<\infty.
    \end{equation}

For the remaining two series, the sum formulas $f^1_{k,t+h}=f^1_{k,t}f^1_{k,h}-f^2_{k,t}f^2_{k,h}$ and $f^2_{k,t+h}=f^2_{k,t}f^1_{k,h}+f^2_{k,h}f^1_{k,t}$ together with the Mean Value Theorem bounds $\left|\frac{f^1_{k,h}-1}{h}\right|,\left|\frac{f^2_{k,h}}{h}\right|\le\left|\frac{k}{M}\right|$ give (similarly as before),
\begin{equation}\label{3.62}
    2\pi\sum_{k=-\infty}^{\infty}(1+k^2)^{s-1}\left|\left(\frac{f^1_{k,t+h}-f^1_{k,t}}{h}+\frac{k}{M}f^2_{k,t}\right)\widehat{U_0}(k)\right|^2\leq\frac{18\pi}{M^2}\sum_{k=-\infty}^{\infty}(1+k^2)^{s}|\widehat{U_0}(k)|^2<\infty,
\end{equation}
\begin{equation}\label{3.63}
    2\pi\sum_{k=-\infty}^{\infty}(1+k^2)^{s-1}\left|-i\left(\frac{f^2_{k,t+h}-f^2_{k,t}}{h}-\frac{k}{M}f^1_{k,t}\right)\widehat{V_0}(k)\right|^2\leq\frac{18\pi}{M^2}\sum_{k=-\infty}^{\infty}(1+k^2)^{s}|\widehat{V_0}(k)|^2<\infty,
\end{equation}
\begin{equation}\label{3.64}
     2\pi\sum_{k=-\infty}^{\infty}(1+k^2)^{s-1}\left|-i\left(\frac{f^2_{k,t+h}-f^2_{k,t}}{h}-\frac{k}{M}f^1_{k,t}\right)\widehat{U_0}(k)\right|^2\leq\frac{18\pi}{M^2}\sum_{k=-\infty}^{\infty}(1+k^2)^{s}|\widehat{U_0}(k)|^2<\infty,  
 \end{equation}
 and
 \begin{equation}\label{3.65}
      2\pi\sum_{k=-\infty}^{\infty}(1+k^2)^{s-1}\left|\left(\frac{f^1_{k,t+h}-f^1_{k,t}}{h}+\frac{k}{M}f^2_{k,t}\right)\widehat{V_0}(k)\right|^2\leq\frac{18\pi}{M^2}\sum_{k=-\infty}^{\infty}(1+k^2)^{s}|\widehat{V_0}(k)|^2<\infty.
 \end{equation}

By \eqref{3.61}, \eqref{3.62}, \eqref{3.63}, \eqref{3.64}, \eqref{3.65}, and the Weierstrass $M$-test, the three series in \eqref{tres series 2} converge uniformly in $h$, so as before we may take the limit $h\to0^+$ term by term; each term cancels exactly, giving
$$\lim _{h \rightarrow 0^+}\left\|\frac{\Psi(t+h)-\Psi(t)}{h}-A \Psi(t)\right\|_{\tilde{Y}}=0.$$

For $0<-h<t$, the same argument used above for \eqref{sun-jin} applies replacing $K_{M,t}$ by $2$ and $C_{\delta,M}$ by $C_{\epsilon,\delta,M}$:
\begin{equation}
\begin{aligned}
         \left\|\frac{\Psi(t+h)-\Psi(t)}{h}-A \Psi(t)\right\|_{\tilde{Y}}  &\leq2\left\|\frac{S(-h)\Psi_0-\Psi_0}{-h}-A\Psi_0\right\|_{\tilde{Y}}+2C_{\epsilon,\delta,M}\left\|\Psi_0- S(-h)\Psi_0\right\|_{Y},
\end{aligned}
\end{equation}
and both terms vanish as $h\to0^-$. Hence $\partial_t\Psi(t)=A\Psi (t)$ in $\tilde{Y}$ for all $t\geq0$, and $\Psi (0)=\Psi_0$, so $\Psi(t)$ is a solution of \eqref{sistema nuevo lineal}; uniqueness follows exactly as in the proof of Theorem \ref{solución de 3.2}.
 \end{proof}
 \medskip
 \medskip
From what was studied above we can deduce the following:
 \begin{corolario}\label{Psi es C1}
$\Psi(t) \in C^1\paren{[0,\infty);\tilde{Y}}$ and the solutions of \eqref{sistema nuevo lineal} depend continuously on the initial data.
 \end{corolario}
 \begin{proof}
     Let $t\geq 0$. By the previous theorem 
\begin{equation}
\begin{aligned}
        \left\|\partial_t\Psi(t+h)-\partial_t\Psi(t)\right\|_{\tilde{Y}}&=\left\|\partial_tS(t+h)\Psi_o-\partial_tS(t)\Psi_o\right\|_{\tilde{Y}}\\
        &=\left\|A(S(t+h)\Psi_o-S(t)\Psi_o)\right\|_{\tilde{Y}}\\
        &\leq C_{\epsilon,\delta,M}\left\|S(t+h)\Psi_o-S(t)\Psi_o\right\|_{Y}.
\end{aligned}
\end{equation}
By continuity of the semigroup,
\begin{equation}
    \lim _{h \rightarrow 0}\left\|\partial_t\Psi(t+h)-\partial_t\Psi(t)\right\|_{\tilde{Y}}\leq \lim _{h \rightarrow 0}C_{\epsilon,\delta,M}\left\|S(t+h)\Psi_o-S(t)\Psi_o\right\|_{Y}=0. 
\end{equation}
Thus, $\Psi(t)$ is continuously differentiable for $t\geq 0$.
\\

Moreover, if $\Psi_{0,1},\Psi_{0,2} \in Y$, and $\Psi_1(t), \Psi_2(t) \in  C\paren{[0,\infty);Y}$ are their associated solutions of \eqref{sistema nuevo lineal},
\begin{equation}
    \sup_{t\in [0,\infty)}\norm{\Psi_1(t)-\Psi_2(t)}_Y=\sup_{t\in [0,\infty)}\norm{S(t)\Psi_{0,1}-S(t)\Psi_{0,2}}_Y\leq 2\norm{\Psi_{0,1}-\Psi_{0,2}}_Y .
\end{equation}
 \end{proof}
With Theorem \ref{S(t)Psi_0 es la solución del sistema nuevo lineal} and Corollary \ref{Psi es C1} we obtain the following result:
 \begin{teorema}\label{Teorema 3.9}
    Let $s\in \mathbb{R}$. The system \eqref{sistema nuevo lineal} is well-posed in $Y=(H^s(\mathbb{T}))^3$. That is, for every $\Psi_0\in Y$, \eqref{sistema nuevo lineal} has a unique solution 
    $$\Psi\in C([0,\infty);Y)\cap C^1([0,\infty);\tilde{Y}=H^{s-2}(\mathbb{T})\times H^{s-2}(\mathbb{T})\times H^{s-2}(\mathbb{T})),$$
    which depends continuously on the initial data.
\end{teorema}

\subsubsection{Local Well-Posedness in $(H^s(\mathbb{T}))^3, s>1/2$}
We now study the local well-posedness of the system \eqref{sistema nuevo}. By Duhamel's principle, we look for a solution of the integral equation
\begin{equation}\label{ecuación integral}
    \Psi(t)=S(t)\Psi_0+\int_0^tS(t-\tau)B(\Psi)(\tau)d \tau.
\end{equation}

\begin{teorema}\label{Phi es contracción}
    Let $\Psi_0 \in Y=\left(H^s(\mathbb{T})\right)^3$, $s>\frac{1}{2}$. There exist a time $T\left(s,\|\Psi_0\|_Y\right)>0$ and a function $\Psi \in C\left([0, T] ; Y\right)$ satisfying \eqref{ecuación integral}.
\end{teorema}

\begin{proof} Let $K=K(\norm{\Psi_0}_Y)=4\norm{\Psi_0}_Y$, and consider the map
\begin{equation}
    \Phi\Psi(t)=S(t)\Psi_0+\int_0^tS(t-\tau)B(\Psi)(\tau)d \tau
\end{equation}
and the space
$$X(T,K):=\left\{\Psi \in C\left([0, T] ; Y\right): \norm{\Psi}:=\sup _{t \in[0, T]}\|\Psi(t)\|_Y \leq K\right\}.$$
$T>0$ will be chosen below. We show that for some $T>0$, $\Phi$ is a contraction on $X(T,K)$, in three steps.

\begin{itemize}
    \item \textbf{Step 1: $\Phi\Psi \in C\left([0, T] ; Y\right)$ whenever $\Psi\in X(T,K)$.} Let $T>0, \Psi\in X(T,K)$, $t \in [0,T]$. Write $C_{\sigma_2, W, D, M}:=\abs{\sigma_2+W\paren{\alpha_1+\frac{D}{M}\alpha_2}}$, $C_{W, D, M}:=W\max\{1,\frac{\abs{D}}{M}\}$. Since $H^s(\mathbb{T})$ is a Banach algebra,
    \begin{equation}\label{norma de B Psi}
        \begin{split}
            \norm{B(\Psi)(t)}_Y&\leq \epsilon \paren{C_{\sigma_2, W, D, M}\norm{|\psi(t)|^2\psi(t)}_s+W\norm{U(t)\psi(t)}_s+W\frac{\abs{D}}{M}\norm{V(t)\psi(t)}_s}\\
            & \leq \epsilon \paren{C_{\sigma_2, W, D, M}C_s^2\norm{\psi(t)}_s^3+C_{W, D, M}C_s\paren{\norm{U(t)}_s+\norm{V(t)}_s}\norm{\psi(t)}_s}\\
            & \leq \epsilon \paren{C_{\sigma_2, W, D, M}C_s^2\norm{\Psi(t)}_Y^3+C_{W, D, M}C_s\norm{\Psi(t)}_Y^2}.
        \end{split}
    \end{equation}
    Taking the supremum over $\tau\in[0,T]$ and using $\Psi\in X(T,K)$,
    \begin{equation}\label{Norma de B Psi tau}
        \sup _{\tau \in[0, T]}\|B(\Psi)(\tau)\|_Y\leq \epsilon \paren{C_{\sigma_2, W, D, M}C_s^2K^3+C_{W, D, M}C_sK^2}=:M_K.
    \end{equation}
    Then, using $\|S(t)\|_{\mathcal B(Y)}\le2$,
    \begin{equation}\label{Norma de PhiPsi}
                \norm{\Phi\Psi(t)}_Y\leq 2\norm{\Psi_0}_Y+2\int_0^t\norm{B(\Psi)(\tau)}_Yd \tau\leq 2\norm{\Psi_0}_Y+2tM_K< \infty,
    \end{equation}
    so $\Phi \Psi (t)\in Y$. For continuity in $t$, write
    \begin{equation}\label{3.75}
        \norm{\Phi\Psi(t)-\Phi\Psi(t')}_Y\leq \norm{S(t)\Psi_0-S(t')\Psi_0}_Y+\norm{\int_0^tS(t-\tau)B(\Psi)(\tau)d \tau-\int_0^{t'}S(t'-\tau)B(\Psi)(\tau)d \tau}_Y.
    \end{equation}
    The first term tends to $0$ as $t'\to t$ by continuity of the semigroup.
    \\
    
    For the second term, if $0\leq t < T$ and $t'\to t^+$, split the difference as
    $$\int_0^t\left(S(t-\tau)-S(t'-\tau)\right)B(\Psi)(\tau)d\tau+\int_t^{t'}S(t'-\tau)B(\Psi)(\tau)d\tau.$$
    The first integral tends to $0$ because $\left(S(t-\tau)-S(t'-\tau)\right)B(\Psi)(\tau)\to0$ pointwise in $\tau$ as $t'\to t^+$, is bounded (by \eqref{Norma de B Psi tau}) uniformly in $\tau$ and $t'$ by $4M_K\in L^1([0,t])$, so by the Dominated Convergence Theorem the integral of the norm tends to $0$. The second integral is bounded by $(t'-t)\cdot 2M_K\to0$. Hence $\Phi\Psi(t')\to\Phi\Psi(t)$ as $t'\to t^+$. The case $0<t\leq T$ and $t'\to t^-$ is entirely analogous, splitting the difference instead as $\int_0^{t'}(S(t-\tau)-S(t'-\tau))B(\Psi)(\tau)d\tau+\int_{t'}^tS(t-\tau)B(\Psi)(\tau)d\tau$. Thus $\Phi\Psi\in C([0,T];Y)$.

    \item \textbf{Step 2: There is $T'>0$ with $\Phi\Psi \in X(T',K)$ for every $\Psi \in X(T',K)$.} Take
    \begin{equation}
        T'= \frac{1}{4\epsilon C_sK \paren{C_{\sigma_2, W, D, M}C_sK+C_{W, D, M}}}.
    \end{equation}
Then, by \eqref{Norma de PhiPsi},
\begin{equation}
        \norm{\Phi\Psi}=\sup _{t \in[0, T']}\norm{\Phi\Psi(t)}_Y \leq 2\norm{\Psi_0}_Y+2T'M_K=\frac{K}{2}+\frac{K}{2}=K,
\end{equation}
using $\|\Psi_0\|_Y=K/4$ and the choice of $T'$.

    \item \textbf{Step 3: $\Phi$ is a contraction on $X(T,K)$ for a suitable $T\in(0,T']$.} Let 
    \begin{equation}\label{escoger T}
        T\in \left(0,T'\right]\cap\paren{0,\frac{1}{2\epsilon C_sK\paren{3C_{\sigma_2, W, D, M}C_sK+C_{W, D, M}}}}
    \end{equation}
and $\Psi_1=(\psi_1,U_1,V_1), \Psi_2=(\psi_2,U_2,V_2) \in X(T,K)$. Using the identities
$$\psi_2U_2-\psi_1U_1=\psi_2\paren{U_2-U_1}+U_1\paren{\psi_2-\psi_1}, \qquad \psi_2V_2-\psi_1V_1=\psi_2\paren{V_2-V_1}+V_1\paren{\psi_2-\psi_1},$$
$$\psi_2\abs{\psi_2}^2-\psi_1\abs{\psi_1}^2=\paren{\abs{\psi_2}^2+\abs{\psi_1}^2}\paren{\psi_2-\psi_1}+\psi_2\psi_1\paren{\overline{\psi_2}-\overline{\psi_1}},$$
and the Banach algebra estimate, we obtain
    \begin{align}
         \left\|B(\Psi_1)(t)\right.&-\left.B(\Psi_2)(t)\right\|_Y=\left\|-i\epsilon\left\{\left(\sigma_2+W\left(\alpha_1+\frac{D}{M}\alpha_2\right)\right)|\psi_1(t)|^2+W\left(U_1(t)+\frac{D}{M}V_1(t)\right)\right\} \psi_1(t)\right. \\
        &\quad + \left.i\epsilon\left\{\left(\sigma_2+W\left(\alpha_1+\frac{D}{M}\alpha_2\right)\right)|\psi_2(t)|^2+W\left(U_2(t)+\frac{D}{M}V_2(t)\right)\right\} \psi_2(t)\right\|_s\\
         &=\left\|i\epsilon\left\{\left(\sigma_2+W\left(\alpha_1+\frac{D}{M}\alpha_2\right)\right)\paren{\psi_2(t)\abs{\psi_2(t)}^2-\psi_1(t)\abs{\psi_1(t)}^2}\right.\right.\\
        &\quad\left.+W\paren{\psi_2(t)U_2(t)-\psi_1(t)U_1(t)}+W\frac{D}{M}\paren{\psi_2(t)V_2(t)-\psi_1(t)V_1(t)}\right\|_s\\
        &\leq \epsilon\left( C_{\sigma_2,W,D,M}C_s^2\paren{\norm{\psi_2(t)}_s^2+\norm{\psi_1(t)}_s^2+\norm{\psi_2(t)}_s\norm{\psi_1(t)}_s}\norm{\psi_2(t)-\psi_1(t)}_s\right.\\
        &\quad +C_{W, D, M}C_s\norm{\psi_2(t)}_s\paren{\norm{U_2(t)-U_1(t)}_s+\norm{V_2(t)-V_1(t)}_s}\\
        &\quad\left.+ C_{W, D, M}C_s \paren{\norm{U_1(t)}_s+\norm{V_1(t)}_s} \norm{\psi_2(t)-\psi_1(t)}_s\right)\\
        &\leq \epsilon C_sK\paren{3C_{\sigma_2,W,D,M}C_sK+C_{W,D,M}}\norm{\Psi_2(t)-\Psi_1(t)}_Y, \label{norma de BPsi1-BPsi2}
    \end{align}
using $\|\psi_i(t)\|_s,\|U_i(t)\|_s,\|V_i(t)\|_s\leq K$ for $\Psi_i\in X(T,K)$.
\\

Therefore
\begin{equation}
        \begin{split}
        \norm{\Phi\Psi_1-\Phi\Psi_2}&\leq \sup _{t \in[0, T]}\int_0^t\norm{S(t-\tau)\paren{B(\Psi_1)(\tau)-B(\Psi_2)(\tau)}}_Yd\tau\\
        &\leq 2T\epsilon C_sK\paren{3C_{\sigma_2,W,D,M}C_sK+C_{W,D,M}}\norm{\Psi_2-\Psi_1}_Y,
    \end{split}
\end{equation}
and by the choice of $T$ in \eqref{escoger T}, the coefficient $2T\epsilon C_sK\paren{3C_{\sigma_2,W,D,M}C_sK+C_{W,D,M}}<1$, so $\Phi$ is a contraction on $X(T,K)$.
\end{itemize}
Since $C\paren{[0, T] ; Y}$ is a Banach space and $X(T,K)$ is a closed ball in it, $X(T,K)$ is itself complete. By the Banach Fixed Point Theorem, there is a unique $\Psi \in X(T,K)$ with $\Phi\Psi=\Psi$.
\end{proof}

\begin{teorema}\label{PVI eq ec int}
    Let $s>\frac{1}{2}$. The IVP \eqref{sistema nuevo} is equivalent to the integral equation \eqref{ecuación integral}: if $\Psi \in C\left([0, T] ; Y\right)$ solves \eqref{sistema nuevo}, then $\Psi$ solves \eqref{ecuación integral}; conversely, if $\Psi \in C\left([0, T] ; Y\right)$ satisfies \eqref{ecuación integral}, then $\Psi$ solves \eqref{sistema nuevo} and $\Psi \in C^1\left([0, T] ;\tilde{Y}\right)$.
\end{teorema}

\begin{proof} Suppose $\Psi \in C\left([0, T] ; Y\right)$ solves \eqref{sistema nuevo}. By Theorem \ref{S(t)Psi_0 es la solución del sistema nuevo lineal} and the product rule for the semigroup (see \cite{pazy1983semigroups}),
$$
\partial_{\tau}\left(S\left(t-\tau\right) \Psi\left(\tau\right)\right)=S\left(t-\tau\right)\left(\partial_{\tau}\Psi\left(\tau\right)-A \Psi\left(\tau\right)\right) = S\left(t-\tau\right) B\Psi\left(\tau\right)
$$
in $\tilde Y$. Integrating in $\tau$ from $0$ to $t$ gives
\begin{equation}
    \Psi(t)=S(t)\Psi_0+\int_0^tS\left(t-\tau\right) B\Psi\left(\tau\right)d\tau,
\end{equation}
so $\Psi$ satisfies the integral equation.
\\

Conversely, suppose $\Psi\in C\left([0, T] ; Y\right)$ satisfies \eqref{ecuación integral}. Fix $t\in[0,T)$ and $h>0$ with $t+h<T$. Write
\begin{equation}\label{int dt+}
    \begin{aligned}
& \frac{1}{h} \int_0^{t+h} S\left(t+h-\tau\right) B\Psi\left(\tau\right) d \tau-\frac{1}{h} \int_0^t S\left(t-\tau\right) B\Psi\left(\tau\right) d \tau \\
& \quad=\frac{1}{h} \int_0^t\left(S\left(t+h-\tau\right)-S\left(t-\tau\right)\right) B\Psi\left(\tau\right) d \tau+\frac{1}{h} \int_t^{t+h} S\left(t+h-\tau\right) B\Psi\left(\tau\right) d \tau.
\end{aligned}
\end{equation}
As in the proof of Theorem \ref{S(t)Psi_0 es la solución del sistema nuevo lineal},
\begin{equation}
\lim _{h \rightarrow 0^{+}} \frac{1}{h} \left(S\left(t+h-\tau\right)-S\left(t-\tau\right)\right) B\Psi\left(\tau\right) =A S\left(t-\tau\right) B\Psi\left(\tau\right), 
\end{equation}

where the limit is uniform in $h$ and is taken in $\tilde{Y}$. Then, by the Dominated Convergence Theorem,
\begin{equation}
 \lim _{h \rightarrow 0^{+}} \frac{1}{h} \int_0^t\left(S\left(t+h-\tau\right)-S\left(t-\tau\right)\right) B\Psi\left(\tau\right) d \tau=\int_0^t A S\left(t-\tau\right) B\Psi\left(\tau\right) d \tau.
\end{equation}
Since $A:Y\subseteq \tilde{Y}\longrightarrow \tilde{Y}$ is a closed operator (its graph is closed, as the graph of any continuous map), the integral of a closed operator applied to a continuous function commutes with the operator, so
\begin{equation}\label{int dt+1}
    \int_0^t A S\left(t-\tau\right) B\Psi\left(\tau\right) d \tau=A \int_0^t S\left(t-\tau\right) B\Psi\left(\tau\right) d \tau.
\end{equation}
For the second term in \eqref{int dt+}, since $\tau\mapsto S(t+h-\tau)B\Psi(\tau)$ is continuous on $[t,t+h]$, the Mean Value Theorem for Riemann-Stieltjes integrals gives some $c\in(t,t+h)$ with
\begin{equation}
    \int_t^{t+h} S\left(t+h-\tau\right) B\Psi\left(\tau\right) d \tau=hS\left(t+h-c\right) B\Psi\left(c\right),
\end{equation}
so, letting $h\to0^+$,
\begin{equation}\label{int dt+2}
    \lim _{h \rightarrow 0^{+}} \frac{1}{h}\int_t^{t+h} S\left(t+h-\tau\right) B\Psi\left(\tau\right) d \tau=B\Psi\left(t\right).
\end{equation}
Combining \eqref{int dt+}, \eqref{int dt+1} and \eqref{int dt+2},
\begin{equation}\label{int dt+ resultado}
 \lim _{h \rightarrow 0^{+}}\frac{1}{h}\paren{\int_0^{t+h} S\left(t+h-\tau\right) B\Psi\left(\tau\right) d \tau-\int_0^t S\left(t-\tau\right) B\Psi\left(\tau\right) d \tau}
=A \int_0^t S\left(t-\tau\right) B\Psi\left(\tau\right) d \tau+B\Psi\left(t\right).
\end{equation}

Now take $t\in(0,T]$ and $h<0$ with $0<t+h$. Write
\begin{equation}\label{int dt-}
\begin{aligned}
    & \frac{1}{h} \int_0^{t+h} S\left(t+h-\tau\right) B\Psi\left(\tau\right) d \tau-\frac{1}{h} \int_0^t S\left(t-\tau\right) B\Psi\left(\tau\right) d \tau \\
& \quad=\frac{1}{h} \int_0^{t+h}\left(S\left(t+h-\tau\right)-S\left(t-\tau\right)\right) B\Psi\left(\tau\right) d \tau-\frac{1}{h} \int_{t+h}^t S\left(t-\tau\right) B\Psi\left(\tau\right) d \tau.
\end{aligned}
\end{equation}
The same Mean Value Theorem argument as above gives
\begin{equation}\label{int dt-2}
    \lim _{h \rightarrow 0^{-}} -\frac{1}{h}\int_{t+h}^t S\left(t-\tau\right) B\Psi\left(\tau\right) d \tau=B\Psi\left(t\right).
\end{equation}
For the other term in \eqref{int dt-}: set $\varepsilon'>0$, by the uniform convergence in Theorem \ref{S(t)Psi_0 es la solución del sistema nuevo lineal} there is $\delta_1>0$ such that $|h|<\delta_1$ implies
\begin{equation}
    \norm{\frac{S\left(t+h-\tau\right)-S\left(t-\tau\right)}{h}B\Psi\left(\tau\right)-A S\left(t-\tau\right) B\Psi\left(\tau\right)}_{\tilde{Y}}<\frac{\varepsilon'}{2T}, \quad \forall \tau \in [0,t+h].
\end{equation}
Also, since $B\Psi\in C([0,T];Y)$ by \eqref{norma de B Psi}, the map $\tau\mapsto AS(t-\tau)B\Psi(\tau)$ is continuous with values in $\tilde Y$, then, by the Mean Value Theorem for Riemann-Stieltjes integrals there is $c\in[t+h,t]$ with 
$$\int_{t+h}^tAS(t-\tau)B\Psi(\tau)d\tau=(-h)AS(t-c)B\Psi(c),$$
so
\begin{equation}
     \norm{\int_{t+h}^tA S\left(t-\tau\right) B\Psi\left(\tau\right) d\tau}_{\tilde{Y}} \leq 2|h|C_{\epsilon,\delta,M}\sup_{\tau\in[t+h,t]}\norm{B\Psi(\tau)}_{Y}.
\end{equation}
Choosing $|h|<\delta_2:=\min\left\{\delta_1,\frac{\varepsilon'}{4C_{\epsilon,\delta,M}\sup_{\tau\in[t+h,t]}\norm{B\Psi(\tau)}_{Y}}\right\}$, 
\begin{equation}
    \begin{split}
    &\norm{\int_0^{t+h}\frac{S\left(t+h-\tau\right)-S\left(t-\tau\right)}{h}B\Psi\left(\tau\right) d\tau-\int_0^tA S\left(t-\tau\right) B\Psi\left(\tau\right) d\tau}_{\tilde{Y}}\\
    &\quad\leq\norm{\int_0^{t+h}\frac{S\left(t+h-\tau\right)-S\left(t-\tau\right)}{h}B\Psi\left(\tau\right) d\tau-\int_0^{t+h}A S\left(t-\tau\right) B\Psi\left(\tau\right) d\tau}_{\tilde{Y}}+\norm{\int_{t+h}^tA S\left(t-\tau\right) B\Psi\left(\tau\right) d\tau}_{\tilde{Y}}\\
    & \quad \leq \int_0^{t+h}\frac{\varepsilon'}{2T} d\tau+2|h|C_{\epsilon,\delta,M}\sup_{\tau\in[t+h,t]}\norm{B\Psi(\tau)}_{Y}\\
    &\quad \leq \frac{\varepsilon'}{2T}(t+h)+\frac{\varepsilon'}{2}<\varepsilon'.
    \end{split}
\end{equation}
so, since $\varepsilon'$ was arbitrary,
\begin{equation}\label{int dt-1}
\lim _{h \rightarrow 0^{-}} \frac{1}{h} \int_0^{t+h}\left(S\left(t+h-\tau\right)-S\left(t-\tau\right)\right) B\Psi\left(\tau\right) d \tau = \int_0^t A S\left(t-\tau\right) B\Psi\left(\tau\right) d \tau =A \int_0^t S\left(t-\tau\right) B\Psi\left(\tau\right) d \tau.
\end{equation}
Combining \eqref{int dt-}, \eqref{int dt-2} and \eqref{int dt-1},
\begin{equation}\label{int dt- resultado}
 \lim _{h \rightarrow 0^{-}}\frac{1}{h}\paren{\int_0^{t+h} S\left(t+h-\tau\right) B\Psi\left(\tau\right) d \tau-\int_0^t S\left(t-\tau\right) B\Psi\left(\tau\right) d \tau}
=A \int_0^t S\left(t-\tau\right) B\Psi\left(\tau\right) d \tau+B\Psi\left(t\right).
\end{equation}
By \eqref{int dt+ resultado} and \eqref{int dt- resultado}, in $\tilde{Y}$,
\begin{equation}\label{int dt resultado}
    \frac{d}{d t} \int_0^t S\left(t-\tau\right) B\Psi\left(\tau\right) d \tau =A \int_0^t S\left(t-\tau\right) B\Psi\left(\tau\right) d \tau+B\Psi(t)=A(\Psi(t)-S(t)\Psi_0)+B\Psi(t).
\end{equation}

Differentiating \eqref{ecuación integral} in $t$ and using \eqref{int dt resultado},
\begin{equation}
    \partial_t\Psi(t)=AS(t)\Psi_0+A(\Psi(t)-S(t)\Psi_0)+B\Psi(t)=A\Psi(t)+B\Psi(t),
\end{equation}
which is continuous in $t$, with $\Psi(0)=\Psi_0$; that is, $\Psi$ solves \eqref{sistema nuevo} and $\Psi \in C^1\left([0, T] ;\tilde{Y} \right)$.
\end{proof}
\medskip

\begin{teorema}\label{BPL sist nuevo}
    Let $s>\frac{1}{2}$. The IVP \eqref{sistema nuevo} is locally well-posed in $Y=\left(H^s(\mathbb{T})\right)^3$: for every $\Psi_0 \in Y$ there exist $T>0$ and a unique $$\Psi \in C\left([0, T] ; Y\right) \cap C^1\left([0, T] ; \tilde{Y}\right)$$ satisfying \eqref{sistema nuevo}. Moreover, the data-to-solution map $\Psi_0 \in Y \rightarrow \Psi \in C\left([0, T] ; Y\right)$ is continuous in the following sense: if $\Psi_{0,n} \in Y$ are such that $\Psi_{0,n} \xrightarrow{Y} \Psi_{0}$ as $n \rightarrow \infty$, and $\Psi_n \in C\left(\left[0, T_n\right] ; Y\right)$ are the solutions constructed in Theorem \ref{Phi es contracción} with $\Psi_n(0)=\Psi_{0,n}$, then for every $T^{\prime} \in\left(0, T\right)$ there is $N\in \mathbb{N}$ such that $n>N$ implies $\Psi_n\in C\left([0, T^{\prime}] ; Y\right)$ and
$$
\lim _{n \rightarrow \infty} \sup _{t \in[0, T^{\prime}]}\left\|\Psi_n(t)-\Psi(t)\right\|_Y=0.
$$
\end{teorema}
 
\begin{proof}
\textbf{\textit{Existence:}} Take $K=4\norm{\Psi_0}_Y$. By Theorem \ref{Phi es contracción}, there exist $T>0$ and a unique $\Psi=(\psi,U,V) \in X(T,K)$ satisfying \eqref{ecuación integral}, so $\Psi \in C\left([0, T] ; Y\right)$. By Theorem \ref{PVI eq ec int}, $\Psi \in C^1\left([0, T] ; \tilde{Y}\right)$ and solves \eqref{sistema nuevo} with $\Psi(0)=\Psi_0$.
\\

\textbf{\textit{Uniqueness:}} Suppose $\Psi_1=(\psi_1,U_1,V_1), \Psi_2=(\psi_2,U_2,V_2)\in C\left([0, T] ; Y\right)$ solve \eqref{sistema nuevo} with data $\Psi_{0,1},\Psi_{0,2}$. By Theorem \ref{PVI eq ec int},
\begin{equation}
    \Psi_1(t)=S(t)\Psi_{0,1}+\int_0^tS(t-\tau)B(\Psi_1)(\tau)d \tau, \quad \Psi_2(t)=S(t)\Psi_{0,2}+\int_0^tS(t-\tau)B(\Psi_2)(\tau)d \tau.
\end{equation}
Let $\mathcal{K}^{*}:=\max\{\kappa_1,\kappa_2,\kappa_3\}$, where
$$\kappa_1=\sup_{t \in[0, T]}\left(C_{\sigma_2,W,D,M}C_s^2\left(\norm{\psi_2(t)}_Y^2+\norm{\psi_1(t)}_Y^2+\norm{\psi_2(t)}_Y\norm{\psi_1(t)}_Y\right)\right.$$
$$\left.+WC_s\norm{U_1(t)}_Y+W\frac{|D|}{M}C_s\norm{V_1(t)}_Y\right),$$
$$\kappa_2=\sup_{t \in[0, T]}\paren{WC_s\norm{\psi_2(t)}_Y},$$
$$\kappa_3=\sup_{t \in[0, T]}\paren{W\frac{|D|}{M}C_s\norm{\psi_2(t)}_Y}.$$ 
Then, using the semigroup bound and and what was seen in \eqref{norma de BPsi1-BPsi2},
\begin{equation}\label{norma Psi1-Psi2}
    \begin{aligned}
        \norm{\Psi_1(t)-\Psi_2(t)}_Y&\leq2\norm{\Psi_{0,1}-\Psi_{0,2}}_Y\\
        &\quad+\int_0^t2\epsilon\left(\kappa_1\norm{\psi_1(\tau)-\psi_2(\tau)}_s+\kappa_2\norm{U_1(\tau)-U_2(\tau)}_s+\kappa_3\norm{V_1(\tau)-V_2(\tau)}_s\right)d\tau\\
        &\leq 2\norm{\Psi_{0,1}-\Psi_{0,2}}_Y+2\epsilon \mathcal{K}^{*}\int_0^t\norm{\Psi_1(\tau)-\Psi_2(\tau)}_Yd\tau.
    \end{aligned}
\end{equation}
By Gronwall's inequality, $\norm{\Psi_1(t)-\Psi_2(t)}_Y\leq2\norm{\Psi_{0,1}-\Psi_{0,2}}_Ye^{2\epsilon \mathcal{K}^{*}t}$ for all $t\in[0,T]$, so $\Psi_{0,1}=\Psi_{0,2}$ gives $\Psi_1=\Psi_2$.
\\

\textbf{\textit{Continuous dependence:}} Let $\varepsilon_1>0$ and $T^{\prime}\in (0,T)$, and choose $\varepsilon_2>0$ with $T^{\prime}+\varepsilon_2<T$. With $K_n=4\norm{\Psi_{0,n}}$ and
\begin{equation}
    g(K_n)=\min\left\{\frac{1}{4\epsilon C_sK_n \paren{C_{\sigma_2, W, D, M}C_sK_n+C_{W, D, M}}}, \frac{1}{2\epsilon C_sK_n\paren{3C_{\sigma_2, W, D, M}C_sK_n+C_{W, D, M}}}\right\},
\end{equation}
the choice of $T$ in \eqref{escoger T} shows $T<g(K)$, and the existence time $T_n$ of each $\Psi_n$ may be chosen with $0<g(K_n)-T_n<\varepsilon_2/2$. Now, since $g$ is a continuous function, there exists $\delta>0$ such that if $|K_n-K_m|<\delta$, then $|g(K_n)-g(K_m)|<\frac{\varepsilon_2}{2}$. Moreover, $\Psi_{0,n} \xrightarrow{Y} \Psi_{0}$, as $n \rightarrow \infty$, so there exists $N_1\in \mathbb{N}$ such that 
$$|K_n-K|\leq 4\norm{\Psi_{0,n}-\Psi_{0}}_Y<\delta, \quad \forall n>N_1.$$
Thus, if $n>N_1$, then
\begin{equation}\label{Kn-k}
    |K_n-K|=\abs{4\norm{\Psi_{0,n}}_Y-4\norm{\Psi_{0}}_Y}\leq 4\norm{\Psi_{0,n}-\Psi_{0}}_Y<\delta,
\end{equation}
whence
$$T-T_n=(T-g(K))+(g(K)-g(K_n))+(g(K_n)-T_n)<0+\frac{\varepsilon_2}{2}+\frac{\varepsilon_2}{2}=\varepsilon_2<T-T^{\prime}.$$
So $T_n>T'$ for $n>N_1$, and, by the uniqueness estimate \eqref{norma Psi1-Psi2} applied with $\Psi_{0,1}=\Psi_{0,n}$, $\Psi_{0,2}=\Psi_0$,
\begin{equation}
    \sup _{t \in[0, T^{\prime}]}\left\|\Psi_n(t)-\Psi(t)\right\|_Y<2\left\|\Psi_{0,n}-\Psi_0\right\|_Ye^{2\epsilon K_n^{*}T^{\prime}},
\end{equation}
where $K_n^*$ is the analogue of $\mathcal K^*$ for the pair $\Psi_n,\Psi$. Since $\sup_{t\in[0,T']}\|\Psi_n(t)\|_Y\le K_n<K+\delta$ and $\sup_{t\in[0,T']}\|\Psi(t)\|_Y\le K$ (by construction in Theorem \ref{Phi es contracción}), $K_n^*$ is bounded above by a constant $K'$ independent of $n$. Taking $N_2\ge N_1$ so that $\|\Psi_{0,n}-\Psi_0\|_Y<\varepsilon_1e^{-2\epsilon K'T'}/2$ for $n>N_2$ gives $\sup_{t\in[0,T']}\|\Psi_n(t)-\Psi(t)\|_Y<\varepsilon_1$ for $n>N_2$. Since $\varepsilon_1>0$ was arbitrary, the claim follows.
\end{proof}

\subsection{Local Well-Posedness of the Modified Zakharov-Rubenchik System in $H^{s}(\mathbb{T})\times H^{s}(\mathbb{T})\times H^{s+1}(\mathbb{T})$, $s>3/2$}
We now prove local well-posedness of \eqref{ZR modificado}, with 
\begin{equation}\label{L_1 L_2}
    L_1=-i\delta\alpha_1\left( \partial_x^2\psi\overline{\psi}- \partial_x^2\overline{\psi}\psi\right), \quad L_2=-\frac{i}{M}\delta\alpha_2\left( \partial_x\psi\overline{\psi}- \partial_x\overline{\psi}\psi\right),
\end{equation}
and $\alpha_1,\alpha_2$ solving 
\begin{equation}\label{alpha1 alpha2}
    \left\{\begin{array}{l}
 \frac{1}{M}\alpha_2+D=-\sigma_3\alpha_1,\\
\frac{1}{M}\alpha_1+M=-\sigma_3\alpha_2.
\end{array}\right.
\end{equation}
Our considerations impose the regimes $M< 1$ and $M>1$.

\begin{teorema}
    Let $s>\frac{3}{2}$ and $M\neq 1$. The IVP \eqref{ZR modificado} is locally well-posed in $Z=H^{s}(\mathbb{T})\times H^{s}(\mathbb{T})\times H^{s+1}(\mathbb{T})$: for every $\tilde{\Psi}_0=(\psi_0, \rho_0, \phi_0) \in Z$ there exist $T>0$ and a unique
    $$\tilde{\Psi}=(\psi, \rho, \phi) \in C\left([0, T] ; Z\right) \cap C^1\left([0, T] ; \tilde{Z}=H^{s-2}(\mathbb{T})\times H^{s-2}(\mathbb{T})\times H^{s-1}(\mathbb{T})\right)$$
    satisfying \eqref{ZR modificado}. Moreover, the data-to-solution map is continuous in the same sense as in Theorem \ref{BPL sist nuevo}.
\end{teorema}

\begin{proof} 
\textbf{\textit{Existence and uniqueness:}} Let $\tilde{\Psi}_0=(\psi_0, \rho_0, \phi_0) \in Z$ and $\Psi_0=(\psi_0, \rho_0-\alpha_1 |\psi_0|^2, M\partial_x\phi_0-\alpha_2 |\psi_0|^2)$. By \eqref{Psi0 in Y}, $\Psi_0 \in Y$, so by Theorem \ref{BPL sist nuevo} there exist $T>0$ and a unique $$\Psi=(\psi, U, V) \in C\left([0, T] ; Y\right) \cap C^1\left([0, T] ; \tilde{Y}\right)$$ satisfying \eqref{sistema nuevo}. Setting $\rho=U+\alpha_1|\psi|^2$, $Q=V+\alpha_2|\psi|^2$, and reversing the computation of Section \ref{deduccion de sistema auxiliar}, we get (using \eqref{igualdad para simplificar}, the formulas for $L_1,L_2$, and the equations satisfied by $\alpha_1,\alpha_2$):
\begin{equation}
    \partial_t\psi-\sigma_3 \partial_x\psi-i\epsilon \delta \partial_x^2\psi+i\epsilon\left\{\sigma_2|\psi|^2+W\left(\rho+\frac{D}{M} Q\right)\right\} \psi=0,
\end{equation}
\begin{equation}
    \begin{split}
        \partial_t\rho &=\partial_t U+\alpha_1\partial_t\paren{\abs{\psi}^2}\\
        &=-\frac{1}{M}\partial_xV+\alpha_1\partial_t\paren{\abs{\psi}^2}\\
         &=-\frac{1}{M}\partial_xQ+\frac{1}{M}\alpha_2\partial_x\paren{\abs{\psi}^2}+D\partial_x\paren{\abs{\psi}^2}-D\partial_x\paren{\abs{\psi}^2}+\alpha_1\partial_t\paren{\abs{\psi}^2}\\
        &=-\frac{1}{M}\partial_xQ-\sigma_3\alpha_1\partial_x\paren{\abs{\psi}^2}+\alpha_1\partial_t\paren{\abs{\psi}^2}-D\partial_x\paren{\abs{\psi}^2}\\
        &=-\frac{1}{M}\partial_xQ+i\epsilon\delta\alpha_1\left( \partial_x^2\psi\overline{\psi}- \partial_x^2\overline{\psi}\psi\right)-D\partial_x\paren{\abs{\psi}^2}\\
        &=-\frac{1}{M}\partial_xQ-D\partial_x\paren{\abs{\psi}^2}-\epsilon L_1,\\
    \end{split}
\end{equation}
and, by the same computation applied to $V$,
\begin{equation}
    \partial_tQ=-\frac{1}{M}\partial_x\rho-M\partial_x\paren{\abs{\psi}^2}-\epsilon M \partial_x L_2.
\end{equation}
So $\psi, \rho, Q$ solve
\begin{equation}\label{ec psi rho y Q}
    \left\{\begin{array}{l}
    \partial_t\psi-\sigma_3 \partial_x\psi-i\epsilon \delta \partial_x^2\psi+i\epsilon\left\{\sigma_2|\psi|^2+W\left(\rho+\frac{D}{M} Q\right)\right\} \psi=0,\\
\partial_t\rho+\frac{1}{M}\partial_xQ+D\partial_x\left(|\psi|^2\right)+\epsilon L_1=0, \\
\partial_tQ+\frac{1}{M} \partial_x\rho+M\partial_x\left(|\psi|^2\right)+\epsilon M \partial_xL_2=0,\\
\psi(t=0)=\psi_0, \quad  \rho(t=0)=\rho_0, \quad Q(t=0)=M\partial_x \phi_0.
\end{array}\right.
\end{equation}

Now, $f\in H^{s}(\mathbb{T})$ if and only if $\partial_x f\in H^{s-1}(\mathbb{T})$, so $\partial_x \psi \in H^{s-1}(\mathbb{T})$, and $\psi, \rho \in H^{s}(\mathbb{T})\subseteq H^{s-1}(\mathbb{T})$. Since $s-1>\frac{1}{2}$, $H^{s-1}(\mathbb{T})$ is a Banach algebra with constant $C_{s-1}$, so
\begin{equation}
\begin{split}
    \norm{\frac{1}{M^2} \rho(t)+|\psi(t)|^2+\epsilon L_2(t)}_{s-1}&\leq \frac{1}{M^2}\norm{ \rho(t)}_{s-1}+C_{s-1}\norm{\psi(t)}^2_{s-1}\\
    &\quad+\frac{\epsilon}{M}\abs{\delta \alpha_2}C_{s-1}\paren{\norm{\partial_x\psi(t)}_{s-1}\norm{\overline{\psi}(t)}_{s-1}+\norm{\partial_x\overline{\psi}(t)}_{s-1}\norm{\psi(t)}_{s-1}}\\
    &<\infty, \quad\forall t\in [0,T].
\end{split}
\end{equation}
Moreover, since $\psi \in C\paren{[0,T];H^{s}(\mathbb{T})}$,
\begin{equation}\label{continuidad de partial psi}
    \norm{\partial_x\psi(t+h)-\partial_x\psi(t)}_{s-1}\leq \norm{\psi(t+h)-\psi(t)}_{s},
\end{equation}
then $\partial_x\psi \in C\paren{[0,T];H^{s-1}(\mathbb{T})}$, and hence $\frac{1}{M^2} \rho+|\psi|^2+\epsilon L_2 \in C\paren{[0,T];H^{s-1}(\mathbb{T})}$. Thus the linear IVP
\begin{equation}
  \left\{\begin{array}{l}
    \partial_t\phi+\frac{1}{M^2} \rho+|\psi|^2+\epsilon L_2=0,\\
    \phi(t=0)=\phi_0 \in H^{s+1}(\mathbb{T}),
    \end{array}\right. 
\end{equation}
has the unique solution
\begin{equation}
\phi(t)=-\int_0^t\paren{\frac{1}{M^2} \rho(\tau)+|\psi(\tau)|^2+\epsilon L_2(\tau)} d\tau + \phi_0.
\end{equation}
Consequently $M\partial_x\phi$ satisfies
\begin{equation}\label{ec Mpartial_xphi}
    \left\{\begin{array}{l}
\partial_tM\partial_x\phi+\frac{1}{M} \partial_x\rho+M\partial_x\left(|\psi|^2\right)+\epsilon M \partial_xL_2=0,\\
M\partial_x\phi(t=0)=M\partial_x \phi_0,
\end{array}\right.
\end{equation}
which is the same linear IVP satisfied by $Q$ in \eqref{ec psi rho y Q}. By uniqueness of that IVP, $Q=M\partial_x\phi$, and by the equivalence $f\in H^s\Leftrightarrow\partial_xf\in H^{s-1}$ again, $\phi \in H^{s+1}(\mathbb{T})$. Substituting $Q=M\partial_x\phi$ back into the first two equations of \eqref{ec psi rho y Q} shows that 
\begin{equation}
    \left\{\begin{array}{l}
       \partial_t\psi-\sigma_3 \partial_x\psi-i\epsilon \delta \partial_x^2\psi+i\epsilon\left\{\sigma_2|\psi|^2+W\left(\rho+D \partial_x \phi\right)\right\} \psi=0,\\
       \partial_t\rho+\partial_x^2\phi+D\partial_x\left(|\psi|^2\right)+\epsilon L_1=0,\\
       \partial_t\phi+\frac{1}{M^2} \rho+|\psi|^2+\epsilon L_2=0,\\
       \psi(t=0)=\psi_0, \quad  \rho(t=0)=\rho_0, \quad \phi(t=0)=\phi_0.
    \end{array}\right.
\end{equation}
That is, $\tilde{\Psi}=(\psi, \rho, \phi) \in C\left([0, T] ; Z\right)$ uniquely solves \eqref{ZR modificado}. Moreover, $L_1 \in H^{s-2}(\mathbb{T})$, since
\begin{equation}
    \norm{L_1}_{s-2}=\norm{-i\delta\alpha_1\partial_x\left( \partial_x\psi\overline{\psi}- \partial_x\overline{\psi}\psi\right)}_{s-2}\leq \abs{\delta\alpha_1}\norm{\partial_x\psi\overline{\psi}- \partial_x\overline{\psi}\psi}_{s-1}<\infty.
\end{equation}
Applying the same two facts (the derivative equivalence and the Banach algebra property, now to each nonlinear term appearing in $\partial_t\tilde\Psi$) shows that $\partial_t \tilde{\Psi}(t)\in \tilde{Z}$ for every $t\in [0,T]$, and, using estimates similar to \eqref{continuidad de partial psi} for each factor, that each component of $\partial_t \tilde{\Psi}$ is a sum of products of functions continuous in $t\in [0,T]$; hence
\begin{equation}
    \tilde{\Psi}=(\psi, \rho, \phi) \in C\left([0, T] ; Z\right) \cap C^1\left([0, T] ; \tilde{Z}\right).
\end{equation}

\textbf{\textit{Continuous dependence:}} Let $T'\in (0,T)$ and $\tilde{\Psi}_{0,n}=(\psi_{0,n}, \rho_{0,n}, \phi_{0,n})\to \tilde{\Psi}_{0}=(\psi_0, \rho_0, \phi_0)$ in $Z$. For each $n$, let $\Psi_{0,n}=(\psi_{0,n}, \rho_{0,n}-\alpha_1 |\psi_{0,n}|^2, M\partial_x\phi_{0,n}-\alpha_2 |\psi_{0,n}|^2)\in Y$ and $\Psi_n=(\psi_n, U_n, V_n)\in C\left([0, T_n] ; Y\right)$ the corresponding solution of \eqref{sistema nuevo} from Theorem \ref{Phi es contracción}. Since
\begin{equation}
    \norm{\partial_x\phi_{0,n}-\partial_x\phi_0}_{s}\leq \norm{\phi_{0,n}-\phi_0}_{s+1}, \qquad \norm{\abs{\psi_{0,n}}^2-\abs{\psi_{0}}^2}_s\leq \paren{\norm{\psi_{0,n}}_s+\norm{\psi_{0}}_s}\norm{\psi_{0,n}-\psi_{0}}_s,
\end{equation}
we get $\Psi_{0,n}\to\Psi_{0}$ in $Y$. Let $\tilde{\Psi}_{n}\in C\left([0, T_n] ; Z\right)$ be the solution of \eqref{ZR modificado} with initial data $\tilde{\Psi}_{0,n}$, constructed as above. By Theorem \ref{BPL sist nuevo}, there is $N\in \mathbb{N}$ such that $n>N$ implies $\Psi_n\in C\left([0, T^{\prime}] ; Y\right)$ and $\sup_{t\in[0,T']}\|\Psi_n(t)-\Psi(t)\|_Y\to0$; since $\rho_n=U_n+\alpha_1 |\psi_n|^2$ and
$$\partial_t \phi_n(t)=-\paren{\frac{1}{M^2} \rho_n(t)+|\psi_n(t)|^2-\epsilon\frac{i}{M}\delta\alpha_2\left( \partial_x\psi_n(t)\overline{\psi_n}(t)- \partial_x\overline{\psi_n}(t)\psi_n(t)\right)},$$
we get, for $n>N$,
\begin{equation}\label{rho_n to rho}
     \psi_n\xrightarrow[n \to \infty]{C([0, T^{\prime}];H^{s}(\mathbb{T}))}\psi, \quad\rho_n\xrightarrow[n \to \infty]{C([0, T^{\prime}];H^{s}(\mathbb{T}))}\rho, \quad \partial_t \phi_n\xrightarrow[n \to \infty]{C([0, T^{\prime}];H^{s-1}(\mathbb{T}))}\partial_t \phi.
\end{equation}
Since this last convergence is uniform in $t\in[0,T']$, integrating in $t$ and using $\phi_{0,n}\to\phi_0$ gives
\begin{equation}
    \lim_{n\to \infty}\phi_n(t)=\int_0^t\partial_t \phi(\tau) d \tau +\phi_0=\phi(t)
\end{equation}
for every $t\in [0,T']$, uniformly, that is,
\begin{equation}\label{phi_n to phi}
    \phi_n\xrightarrow[n \to \infty]{C([0, T^{\prime}];H^{s+1}(\mathbb{T}))} \phi.
\end{equation}
From \eqref{rho_n to rho} and \eqref{phi_n to phi}, $\sup_{t\in[0,T']}\|\tilde\Psi_n(t)-\tilde\Psi(t)\|_Z\to0$.
\end{proof}
\medskip
\begin{remark}
    By the choice of $T$ in \eqref{escoger T}, the existence time in the local well-posedness of \eqref{ZR modificado} is of order $O(\epsilon^{-1})$. 
\end{remark}

\newpage
\section{Open Problems}\label{problemas abiertos}

\begin{enumerate}
    \item Global existence of solutions of the modified Zakharov-Rubenchik system (periodic case), or finite-time blow-up.

    \item In two or three dimensions, existence of local/global-in-time solutions on the correct time scale $O(\epsilon^{-1})$, fully justifying the Benney-Roskes system as a water wave model.

    \item Extending the Schochet-Weinstein method (used in \cite{SchochetWeinstein1986}, \cite{obrecht2015approximation}, \cite{luong2018cauchy}), which does not rely on any dispersive property of the Schrödinger group, to the Cauchy problem on $\mathbb{T}^d$ or $\mathbb{R}^{d-1} \times \mathbb{T}$, $d=2,3$, under weaker regularity assumptions on the data.
\end{enumerate}

\bibliographystyle{plain}  
\cleardoublepage
\phantomsection
\addcontentsline{toc}{section}{Bibliography}

\bibliography{bibliografia}

\end{document}